\documentclass[11pt]{article}

\usepackage[authoryear]{natbib}

\usepackage[latin1]{inputenc}
\usepackage[T1]{fontenc}
\usepackage{color}
\usepackage{eurosym}
\usepackage[top=4cm,bottom=4cm,left=3cm,right=3cm]{geometry}
\usepackage{fancyhdr}
\usepackage{float}
\usepackage{amssymb,amsmath, amsthm}
\usepackage{url,graphicx}

\usepackage{booktabs}
\usepackage{stmaryrd} 
\usepackage{mathenv}
\usepackage{dsfont}

\usepackage{amssymb}
\usepackage{latexsym}
\usepackage{amsmath}
\usepackage{color}
\usepackage{comment}
\usepackage{amsthm}
\usepackage{bm}
\usepackage{bbm} 
\usepackage[bbgreekl]{mathbbol} 
\usepackage{authblk}
\usepackage{multirow}
\usepackage{enumerate}   

\newif\ifrs
\rstrue
\ifrs \usepackage{mathrsfs} \fi  
\newif\ifcol
\coltrue 

\newtheorem{theorem*}{Theorem}

\newtheorem{note*}[theorem*]{Note}
\newtheorem{lemma*}[theorem*]{Lemma}
\newtheorem{definition*}[theorem*]{Definition}
\newtheorem{proposition*}[theorem*]{Proposition}
\newtheorem{corollary*}[theorem*]{Corollary}
\newtheorem{remark*}[theorem*]{Remark}
\newtheorem{example*}[theorem*]{Example}

\newif\ifcol
\coltrue
\ifcol
\newcommand{\colorr}{\color[rgb]{0.8,0,0}}

\newcommand{\colorn}{\color[rgb]{1,1,1}}

\else

\newcommand{\colorr}{\color{black}}

\newcommand{\colorn}{\color{black}}

\fi
\excludecomment{en-text}
\includecomment{jp-text}
\includecomment{comment}
\def\bd{\begin{description}}
\def\ed{\end{description}}

\def\D2{\bbD_{2,\infty-}}

\def\D{{\bf D}}

\def\F{{\bf F}}

\def\calb{{\cal B}}

\def\calf{{\cal F}}
\def\calg{{\cal G}}
\def\calh{{\cal H}}

\def\calm{{\cal M}}
\def\caln{{\cal N}}

\def\half{\frac{1}{2}}

\def\be{\begin{equation}}
\def\ee{\end{equation}}
\def\bea{\begin{eqnarray}}
\def\eea{\end{eqnarray}}
\def\beas{\begin{eqnarray*}}
\def\eeas{\end{eqnarray*}}
\def\bi{\begin{itemize}}
\def\ei{\end{itemize}}

\def\bd{\begin{description}}
\def\ed{\end{description}}
\def\l{\left}
\def\r{\right}

\newcommand{\bbD}{{\mathbb D}}

\newcommand{\reels}{\mathbb{R}}
\newcommand{\naturels}{\mathbb{N}}

\newcommand{\esp}{\mathbb{E}}
\newcommand{\proba}{\mathbb{P}}

\newcommand{\inv}[1]{\frac{1}{#1}}

\newcommand{\ind}[1]{\mathbf{1}_{\l\{#1\r\}}}

\usepackage{lipsum}

\begin{document}
 \title{Estimation for high-frequency data under parametric market microstructure noise}
\author{Simon Clinet\footnote{Faculty of Economics, Keio University. 2-15-45 Mita, Minato-ku, Tokyo, 108-8345, Japan. Phone:  +81-3-5427-1506. E-mail: clinet@keio.jp website: http://user.keio.ac.jp/\char`\~clinet/}   and Yoann Potiron\footnote{Faculty of Business and Commerce, Keio University. 2-15-45 Mita, Minato-ku, Tokyo, 108-8345, Japan. Phone:  +81-3-5418-6571. E-mail: potiron@fbc.keio.ac.jp website: http://www.fbc.keio.ac.jp/\char`\~ potiron}}
\date{This version: \today}

\maketitle

\begin{abstract}
We develop a general class of noise-robust estimators based on the existing estimators in the non-noisy high-frequency data literature. The microstructure noise is a parametric function of the limit order book. The noise-robust estimators are constructed as plug-in versions of their counterparts, where we replace the efficient price, which is non-observable, by an estimator based on the raw price and limit order book data. We show that the technology can be applied to five leading examples where, depending on the problem, price possibly includes infinite jump activity and sampling times encompass asynchronicity and endogeneity. 
\end{abstract}
\\
\section{Introduction}

It is now widely acknowledged that the availability of high-frequency data has led to a more accurate description of financial markets. Over the past decades, empirical studies have unveiled several aspects of the frictionless efficient price. Accordingly, the assumptions on the latter have been gradually weakened to the extent that it is common nowadays to represent it as a general It\^{o} semi-martingale including jumps with infinite activity. Moreover, the sampling times are also often considered as asynchronous, random, and even sometimes endogenous, i.e. possibly correlated with the efficient price. The accessibility of high-frequency data has also shed light on the frictions, or so-called market microstructure noise (MMN), which get prominent as the sampling frequency increases. As a matter of fact, realized volatility (i.e. summing the square returns), which is efficient in the absence of frictions becomes badly biased when the frequency increases. This was visible on the signature plot in \cite{andersen2001bdistribution}. A typical challenge that faces a theoretical statistician today is to incorporate jumps, asynchronicity, endogeneity and frictions into the model. 

\smallskip
A frequently used set-up is 
\begin{eqnarray}
\label{frequentsetup}
\underbrace{Z_{t_i}}_{\text{observed price}}  =  \underbrace{X_{t_i}}_{\text{efficient price}} +  \underbrace{\epsilon_{t_i}}_{\text{MMN}},
\end{eqnarray}
    where $\epsilon_{t_i}$ is i.i.d. and latent. In two nice and independent papers, \cite{li2016efficient} and \cite{chaker2017high}, and subsequently \cite{clinet2019testing} and \cite{clinet2019disentangling}, consider the following parametric form for the noise to estimate volatility:
\begin{eqnarray}  
\label{genmodel}
    \underbrace{Z_{t_i}}_{\text{observed price}}  =  \underbrace{X_{t_i}}_{\text{efficient price}} + \underbrace{\underbrace{\phi(Q_{t_i}, \theta_0)}_{\text{parametric noise}} + \underbrace{\epsilon_{t_i}}_{\text{residual noise}}}_{\text{MMN}},
\end{eqnarray}
where $Q_{t_i}$ is the information from the limit order book and $\phi$ is a function known to the statistician. A simple and familiar example was introduced in \cite{roll1984simple}, and specified in e.g. \cite{hasbrouck2002stalking}, where 
\begin{eqnarray}
\label{roll}
\phi(Q_{t_i}, \theta_0) = I_{t_i} \theta_0,
\end{eqnarray}
with $I_{t_i}$ corresponding to the trade direction, i.e. 1 if the transaction at time $t_i$ is buyer initiated and -1 if seller initiated, and $\theta_0$ standing for half of the effective spread. In \cite{glosten1988estimating}, the extension includes the trading volume $V_{t_i}$ and takes on the form 
\begin{eqnarray}
\label{glostenharris}
\phi(Q_{t_i}, \theta_0) = I_{t_i} \theta_0^{(1)} + I_{t_i} V_{t_i} \theta_0^{(2)}. 
\end{eqnarray}
A different model features information about the quoted spread $S_{t_i}$, where 
\begin{eqnarray}
\label{spreadmodel}
\phi(Q_{t_i}, \theta_0) = I_{t_i} S_{t_i} \theta_0.
\end{eqnarray}
This model can be seen as an updated time-varying Roll model, as the quoted spread is nowadays available in the structure of current limit order book markets, whereas it was not observed at the time when Roll model was proposed.

\smallskip
There are two regimes related to the parametric model (\ref{genmodel}), i.e. the null residual noise and non zero residual noise. To estimate volatility, the cited papers rely on a plug-in procedure. In a first step, they provide estimators of the parameter $\theta_0$ and establish fast convergence rate which satisfies
\begin{eqnarray}
\label{ratethetahat}
N (\widehat{\theta} - \theta_0) = O_\proba (1),
\end{eqnarray}
where $N$ stands for the number of observations and pre estimate the efficient price via
\begin{eqnarray}
\widehat{X}_{t_i}  =  Z_{t_i} - \phi(Q_{t_i}, \widehat{\theta}).
\label{priceestimator}
\end{eqnarray} 
In a second step, one can apply a "usual" estimator of volatility, considering the observed price as in fact the pre estimated efficient price. More specifically, in case of absence of residual noise, the cited papers implement realized volatility and retrieve efficiency of the method. In the presence of residual noise, they also provide residual noise robust estimators.

\smallskip
In this paper, we will assume the null residual noise regime, which we agree is quite a strong assumption (at first glance). Indeed, from a theoretical statistician standpoint, the non zero residual noise regime, of which the common set-up (\ref{frequentsetup}) is a particular case, is obviously more challenging.  Nonetheless, the original papers \cite{li2016efficient} and \cite{chaker2017high} most likely wanted to select empirically variables from the limit order book that fully explain the MMN. Actually, in their empirical study on four stocks and one day, \cite{li2016efficient} find that the residual noise of models such as (\ref{roll}) and (\ref{glostenharris}) accounts for 20-30\% of the total MMN variance, which is quite low and yet not negligible. \cite{chaker2017high} proposes and implements on a full year of one stock from the New York Stock Exchange tests for the absence of residual noise. She finds rejection rate around 15-25\% for (\ref{roll}), and 10-30\% in the case of (\ref{glostenharris}), here again quite nice results but not indicating the absence of residual noise. More recently, implemented on a month with 31 constituents from the CAC 40, \cite{clinet2019testing} find that the "best" model among several competitors from the financial economics literature is (\ref{spreadmodel}), with related residual noise accounting for as low as 1\% of the MMN variance, and results in line with previous findings for the other models. Finally, in an extensive study on 50 stocks randomly selected from the S\&P 500 during the period 2009-2017, \cite{clinet2019disentangling} exhibit (\ref{spreadmodel}) as the model explaining the most variance of the MMN, with residual noise accounting for (almost) 0\% of the total MMN variance. Those two empirical studies back up the null residual noise regime.

\smallskip
When implementing a non noise-robust procedure with high frequency data, it is often the case that the applied statistician faces a dilemma in using tick-by-tick data on the statistical principle that one should not throw away data, or subsampling -say every five minutes- in respect to the limited theoretical assumptions. We argue that the plug-in approach is a cheap method that kills two birds with one stone. On the one hand, it provides the theoretical statistician with a simple and transparent method for adding MMN in his theory. On the other hand, this will be useful for the applied statistician as he/she will be able to use tick-by-tick data when implementing the related estimator. This strategy is actually successfully used in \cite{andersen2019time}. In particular, our paper enlightens the theoretical aspect of the plug-in approach.

\smallskip
To do so, we describe the general framework as follows. If we define the horizon time as $T$, one typically seeks to estimate the random integrated parameter
\begin{eqnarray}
\label{object0}
\Xi = \int_0^T \xi_t dt,
\end{eqnarray}
where the spot parameter $\xi_t$ is a stochastic process which can correspond to the volatility, the high-frequency covariance, functionals of volatility and volatility of volatility, employing a given data-based estimator $\widetilde{\Xi}(X_{t_0}, \cdots, X_{t_N})$. In the absence of noise, $\widetilde{\Xi}$ usually enjoys a stable central limit theorem of the form
\begin{eqnarray}
\label{convergence}
N^{\kappa} \big(\widetilde{\Xi} - \Xi \big) & \rightarrow & \mathcal{MN}\big(AB,AVAR \big),
\end{eqnarray}
where $\kappa > 0$ corresponds to the rate of convergence, and $\calm\caln(AB,AVAR)$ designates a mixed normal distribution of random bias $AB$ and random variance $AVAR$ (due to the fact that the parameter itself is random). In addition, for the purpose of practical implementation, one typically provides a related studentized central limit theorem, i.e. data-based statistics $\widetilde{AB}(X_{t_0}, \cdots, X_{t_N})$ and $\widetilde{AVAR}(X_{t_0}, \cdots, X_{t_N})$ such that
\begin{eqnarray}
\label{studentizedconvergence}
N^{\kappa} \frac{\widetilde{\Xi} - N^{- \kappa} \widetilde{AB} - \Xi}{\sqrt{\widetilde{AVAR}}} & \rightarrow & \mathcal{N}(0,1).
\end{eqnarray}

\smallskip
Accordingly, when observations are contaminated by the parametric noise, we propose to exploit the corresponding class of plug-in estimators to estimate the integrated parameter. They are constructed as $\widehat{\Xi} =\widetilde{\Xi}(\widehat{X}_{t_0}, \cdots, \widehat{X}_{t_N})$, $\widehat{AB} = \widetilde{AB}(\widehat{X}_{t_0}, \cdots, \widehat{X}_{t_N})$ and $\widehat{AVAR} = \widetilde{AVAR}(\widehat{X}_{t_0}, \cdots, \widehat{X}_{t_N})$.
This plug-in approach seems to be traced back to the framework of the model with uncertainty zones from \cite{robert2010new} and \cite{robert2012volatility}.

\smallskip
The main contribution of this paper is presented in Section 4, where we state that under parametric noise the central limit theorems (\ref{convergence}) and (\ref{studentizedconvergence}) still hold when we substitute the estimators by their related plug-in version in five leading examples of the literature. Depending on the problem at hand, price possibly features jumps with infinite activity and sampling times include asynchronicity and endogeneity. The first example considers the threshold realized volatility inspired by \cite{andersen2001distribution}, \cite{barndorff2002estimating} and \cite{mancini2009non}. Technically, we extend the central limit theory of realized volatility under endogenous sampling in \cite{li2014realized}, which includes no jumps to allow for jumps with infinite activity. The second example deals with the threshold bipower variation, which was originally with no threshold in \cite{barndorff2004power}, and from \cite{corsi2010threshold} and \cite{vetter2010limit}. In the third example, we discuss the \cite{hayashi2005covariance} estimator to estimate high-frequency covariance. The fourth example is devoted to the local estimator from \cite{jacod2013quarticity} which estimates functionals of volatility. Finally, we focus on the estimator of volatility of volatility introduced in \cite{vetter2015estimation} in the last example. 

\smallskip
In all those examples, the only required assumption on $\widehat{\theta}$ to obtain (\ref{convergence}) and (\ref{studentizedconvergence}) is the fast convergence (\ref{ratethetahat}), which is already obtained in a general setting where price process features big jumps in \cite{li2016efficient}, so that our contribution in that respect boils down to adding possible small jumps. Moreover, the asymptotic properties in both equations remain unchanged, whereas the rate of convergence is slower in the i.i.d latent noise case. It means that the parametric noise assumption induces faster rates of convergence than the i.i.d condition, but it is fair to say that we play a different game in this paper as plug-in estimators exploit supplementary data available from the limit order book.

\smallskip
The rest of this paper is structured as follows. Section 2 introduces the model. Section 3 is devoted to the estimation. The five examples are developed in Section 4. We conclude in Section 5. Proofs can be found in Section 6.

\section{Model}
Almost all the quantities defined in what follows are multi-dimensional. Accordingly, the notation $x^{(k)}$ refers to the $k$-th component of $x$. We define the horizon time as $T >0$, and the (possibly random) number of observations\footnote{All the defined quantities are implicitly or explicitly indexed by $n$ (except for the integrated parameter which does not depend on $n$). For example $N$ should be thought and considered as $N_n$. Consistency and convergence in law refer to the behavior as $n \rightarrow \infty$. A full specification of the model also involves the stochastic basis $\calb=(\Omega,\proba,\calf,\F)$, where $\calf$ is a $\sigma$-field and $\F=(\calf_t)_{t\in[0,T]}$ is a filtration, which will be example-specific. We assume that all the processes (including the integrated parameter $\xi_t$) are $\F$-adapted (either in a continuous or discrete meaning for $Q_{t_i}$) and that the observation times $t_i$ are $\F$-stopping times. Also, when referring to It\^{o}-semimartingale and stable convergence in law, we automatically mean that the statement is relative to $\F$. Finally, we assume in (\ref{efficientPrice}) that $W$ is also a Brownian motion under the larger filtration $\calh_t = \calf_t \vee \sigma\{Q_{t_i}, 0\leq i\leq N\}$.} as $N$. The observation times, which satisfy $0 \leq t_0^{(k)} \leq ...\leq t_{N}^{(k)} \leq T $, are possibly asynchronous, i.e. they may differ from one price component to the next (see Section \ref{covariancesec}), and endogenous, i.e. correlated with $X_t$ (as in Section \ref{volatilitysec} and Section \ref{covariancesec}). When observations are regular and synchronous, we have $\Delta_i t :=  t_i - t_{i-1} = T/n := \Delta$ (as in Section \ref{bipowersec}, Section \ref{functionalssec} and Section \ref{sectionVolOfVol}), which implicitly means that $N=n$ and $t_i$ are 1-dimensional, although the price process can be multi-dimensional. 

\smallskip
In view of the empirical findings described in the introduction, it is natural to specify (\ref{genmodel}) as the "pure" parametric noise model via
\begin{eqnarray}
\underbrace{Z_{t_i}}_{\text{observed price}}  =  \underbrace{X_{t_i}}_{\text{efficient price}} + \underbrace{\phi(Q_{t_i}, \theta_0)}_{\text{parametric noise}},
\label{eqDecompositionAdditive}
\end{eqnarray}
where the parameter $\theta_0 \in \Theta \subset \reels^l$ with $\Theta$ a compact set, the impact function $\phi$ is known of class $C^3$ in its second argument, and $Q_{t_i} \in \reels^q$ includes \textit{observable} information\footnote{Note that we do not assume that $Q_t$ exists for any $t \in [0,T] - \{t_0, \cdots, t_N\}$ as it is often the case in the i.i.d setting, see, e.g., the framework in \cite{jacod2009microstructure}.} at the observation time $t_i$ from the limit order book such as the aforementioned trade type (\cite{roll1984simple}), trading volume  (\cite{glosten1988estimating}) and quoted bid-ask spread, but also possibly the duration time between two trades (\cite{almgren2001optimal}), the quoted depth (\cite{kavajecz1999specialist}), the order flow imbalance (\cite{cont2014price}), etc. In practice, $\phi$ could be always chosen as (\ref{spreadmodel}), although we do not specify this particular model in the paper for generality purposes. Further discussion is available in: \cite{black1986noise}, \cite{hasbrouck1993assessing}, \cite{o1995market}, \cite{madhavan1997security}, \cite{madhavan2000market}, \cite{stoll2000presidential} and \cite{hasbrouck2007empirical} among other prominent works. One can also look at the review from \cite{diebold2013correlation}. Finally, on the grounds that the one-lag autocorrelation in mid price returns is often found positive empirically, \cite{andersen2017volatility} extend the usual martingale-plus-noise setting to allow for positivity in the one-lag serial autocorrelation. Note that the model of (\ref{eqDecompositionAdditive}), without residual noise, is theoretically interesting because it allows to adapt the existing methods by plugging in the estimated price in place of the existing estimator.

\smallskip
Finally, we assume that  
\begin{eqnarray}
\label{infoass}
\max_{i,j,k} \big| Q_{t_i^{(k)}}^{(k,j)} \big| = O_\proba (1),
\end{eqnarray}
where $Q_{t_i^{(k)}}^{(k)} = (Q_{t_i^{(k)}}^{(k,1)}, \cdots, Q_{t_i^{(k)}}^{(k,j_k)})$ corresponds to the information related to $X^{(k)}$ at time $t_i^{(k)}$. The latent $d$-dimensional log-price $X_t$ possibly including jumps and its related $d^2$-dimensional spot volatility $c_t = \sigma_t \sigma_t^T$ are It\^{o}-semimartingales of the form 
\begin{eqnarray}
\label{efficientPrice} X_t & = & X_0 + \int_0^t b_s ds + \int_0^t \sigma_s dW_s + \int_0^t \int_\reels \delta (s,z) \mathbf{1}_{\{\mid \mid \delta (s,z) \mid \mid \leq 1\}} (\mu - \nu) (ds, dz)\\
\nonumber & & + \int_0^t \int_\reels \delta (s,z) \mathbf{1}_{\{\mid \mid \delta (s,z) \mid \mid > 1\}} \mu (ds, dz),\\\label{volatilityDefinition}
c_t & = & c_0 + \int_0^t \widetilde{b}_s ds + \int_0^t \widetilde{\sigma}_s dW_s' + \int_0^t \int_\reels \widetilde{\delta} (s,z) \mathbf{1}_{\{\mid \mid \widetilde{\delta} (s,z) \mid \mid \leq 1\}} (\mu - \nu) (ds, dz)\\
\nonumber & & + \int_0^t \int_\reels \widetilde{\delta} (s,z) \mathbf{1}_{\{\mid \mid \widetilde{\delta} (s,z) \mid \mid > 1\}} \mu (ds, dz),
\end{eqnarray}
where $W_t$ is a $d$-dimensional Brownian motion and $W_t'$ is a $d^2$-dimensional Brownian motion possibly correlated with $W_t$, the $d$-dimensional $b_t$ and $d^2$-dimensional $\widetilde{b}_t$ drifts are locally bounded, $\sigma_t$ and the $d^2$-dimensional $\widetilde{c}_t = \widetilde{\sigma}_t \widetilde{\sigma}_t^T$ are locally bounded, $\mu$ is a Poisson random measure on $\reels^{+} \times E$ where $E$ is an auxiliary Polish space, with the related intensity measure, i.e. the nonrandom predictable compensator, $\nu(dt,dz) = dt \otimes \lambda (dz)$ for some $\sigma$-finite measure $\lambda$ on $\reels^+$. Finally, $\delta = \delta(\omega,t,z)$ (respectively $\widetilde{\delta}$) is a predictable $\reels^d$-valued ($\reels^{d \times d}$-valued) function on $\Omega \times \reels^+ \times \reels$ such that locally $\sup_{\omega,t} \mid \mid \delta(\omega,t,z) \mid \mid^r \leq \gamma(z)$ ($\sup_{\omega,t} \mid \mid \widetilde{\delta}(\omega,t,z) \mid \mid^{\widetilde{r}} \leq \gamma(z)$) for some nonnegative bounded $\lambda$-integrable function $\gamma$ and some\footnote{Here the restriction $r < 1$ follows from \cite{jacod2013quarticity}. Indeed, even for the realized volatility problem, (\ref{convergence1}) may not happen in the case $r > 1$. Indeed, it yields a different optimal rate of convergence as shown in \cite{jacod2014remark} (of the form $N^\kappa \textnormal{log}N$ for some $\kappa>0$). Moreover, as explained in their Remark 3.4, a CLT is not even achievable in some cases. The case $r=1$ is let aside. Such bordercase is examined in \cite{vetter2010limit} when considering the bipower variation.} $r \in [0,1)$ ($\widetilde{r} = 2$). Furthermore, we define the "genuine" drift as $b_t' = b_t - \int\delta(t,z) \mathbf{1}_{ \{ \mid \mid \delta(t,z) \mid \mid \leq 1 \}}\lambda(dz)$, the continuous part of $X_t$ as 
$$X_t' =  X_0 + \int_0^t b_s' ds + \int_0^t \sigma_s dW_s,$$
and the jump part as $J_t = \sum_{s \leq t} \Delta X_s$. Key to our analysis is the decomposition 
\begin{eqnarray}
\label{decompoX0}
X_t = X_t' + J_t.
\end{eqnarray}

\section{Estimation under parametric noise}
\subsection{Integrated parameter estimation}
The object of interest can be the integrated volatility, etc. In the non-noisy version of the problem, the typical scenario is such that the high-frequency data user has a data-based estimator $\widetilde{\Xi} (X_{t_0}, \cdots, X_{t_N})$ of (\ref{object0}),  such as the standard realized volatility (RV), i.e. $RV = \sum_{i=1}^N \Delta_i X^2$ where $\Delta_i A = A_{t_i} - A_{t_{i-1}}$,  and possibly a related central limit theorem and a studentized version of it. In all generality, they respectively take the form of 
\begin{eqnarray}
\label{convergence1}
N^{\kappa} \big(\widetilde{\Xi} - \Xi \big) & \rightarrow & \mathcal{MN} \big(AB,AVAR \big),
\end{eqnarray}
where $\kappa > 0$ corresponds to the rate of convergence, and 
\begin{eqnarray}
\label{studentizedconvergence1}
N^{\kappa} \frac{\widetilde{\Xi} - N^{- \kappa} \widetilde{AB} - \Xi}{\sqrt{\widetilde{AVAR}}} & \rightarrow & \mathcal{N}(0,1),
\end{eqnarray}
where $\widetilde{AB}(X_{t_0}, \cdots, X_{t_N})$ and $\widetilde{AVAR}(X_{t_0}, \cdots, X_{t_N})$ are also data-based statistics which respectively correspond to the asymptotic bias and the asymptotic variance estimator. The aim of this section is to equip the high-frequency data user with noise-robust estimators which are based on $\widetilde{\Xi}$.

\smallskip
To estimate the integrated parameter, we first need an estimator of the noise parameter $\theta_0$ defined as $\widehat{\theta}$. We assume that $\widehat{\theta}$ satisfies
\begin{eqnarray}
\label{thetahat}
N (\widehat{\theta} - \theta_0) = O_\proba (1).
\end{eqnarray}
The techniques of this paper are estimator independent and only require (\ref{thetahat}). In Section \ref{parameterestimation}, we provide the form of the estimators from the literature which satisfy (\ref{thetahat}) (see Proposition \ref{thmYingYing} below). Based on $\widehat{\theta}$, the efficient price is naturally estimated as
\begin{eqnarray}
\widehat{X}_{t_i}  =  Z_{t_i} - \phi(Q_{t_i}, \widehat{\theta}).
\end{eqnarray} \label{estimatedPriceFormula}
This estimator was already used in \cite{li2016efficient}, \cite{chaker2017high} and \cite{clinet2019testing}. The related plug-in estimator is constructed as 
\begin{eqnarray}
\widehat{\Xi} =\widetilde{\Xi}(\widehat{X}_{t_0}, \cdots, \widehat{X}_{t_N}).
\end{eqnarray}
For instance, in the case of RV, we obtain that $\widehat{RV} = \sum_{i=1}^N \Delta_i \widehat{X}^2$. Similarly, we introduce $\widehat{AB} = \widetilde{AB}(\widehat{X}_{t_0}, \cdots, \widehat{X}_{t_N})$ and $\widehat{AVAR} = \widetilde{AVAR}(\widehat{X}_{t_0}, \cdots, \widehat{X}_{t_N})$. 

\smallskip
We end this section with a succinct remark on the theoretical implications of (\ref{thetahat}). At this point, the reader may notice that the fast rate $N^{-1}$ in (\ref{thetahat}), which implies the approximation $\widehat{X}_{t_i} = X_{t_i} + \psi_i(\widehat{\theta})$ with $\psi_i(\widehat{\theta}) = \phi(Q_{t_i}, \theta_0)-\phi(Q_{t_i}, \widehat{\theta}) = O_\proba(N^{-1})$ by (\ref{estimatedPriceFormula}), suggests that the perturbation $\psi_i(\widehat{\theta})$ acts as an additional drift component and therefore could be systematically treated as such in all derivations. There is, however, a fundamental difference between the two quantities, in that drift returns $\Delta_i B = \int_{t_{i-1}}^{t_{i}}b_sds$ are typically adapted, hence $\calf_{t_i}$ measurable, whereas $\psi_i(\widehat{\theta})$, through $\widehat{\theta}$, depends not only on the additional observations $(Q_{t_j})_{j = 0,...,N}$ but also on the whole trajectory of the price process $X$, that is $\calf_T$. This may pose a problem when considering, for instance, terms of the form $A_{i-1} \Delta_i M$ where $\Delta_i M$ is a martingale increment (even for the augmented filtration $\calh_t = \calf_t \vee \sigma \{Q_{t_i}, 0\leq i \leq N\}$). Indeed, when $A_{i-1} = \Delta_{i-1} B$, it naturally preserves the martingale structure of $A_{i-1} \Delta_i M$. On the other hand, if $A_{i-1} = \psi_{i-1}(\widehat{\theta})$, such a structure is broken, and additional arguments are necessary in order to retrieve the desired order of the increment $A_{i-1} \Delta_i M$. In this simple example, the problem can be circumvented with a Taylor expansion $\psi_{i-1}(\widehat{\theta}) \approx  (\widehat{\theta}-\theta_0)^T \partial_\theta\psi_{i-1}(\theta_0) + r_{i-1}(\widehat{\theta})$, using that now $\partial_\theta \psi_{i-1}(\theta_0)$ is $\calh_{t_{i-1}}$ measurable, and that $ r_{i-1} (\widehat{\theta})$ is of order $N^{-2}$.


\subsection{Noise parameter estimation}
\label{parameterestimation}
Several estimators have been proposed by \cite{li2016efficient}, \cite{chaker2017high}, \cite{clinet2019testing} in different settings when we assume a null residual noise $\epsilon_t = 0$. The estimator from \cite{chaker2017high} coincides with the minimum mean square error (MSE) estimator from \cite{li2016efficient} when $\phi$ is linear (which is the related assumption of the former paper). Moreover, the quasi maximum likelihood estimation (QMLE) from \cite{clinet2019testing} reduces to the MSE, due to the Gaussian form of the quasi likelihood function. Accordingly, we review the MSE procedure below and give the related limit theory for the noise parameter estimator.  

\smallskip
We assume that $\theta:= (\theta_0^{(1)}, \ldots, \theta_0^{(d)})$, where for each component $k=1, \cdots,d$ we have $\theta_0^{(k)} := (\theta_0^{(k,1)}, \ldots, \theta_0^{(k,l_k)})$, which corresponds to the parameter related to the $k$th component of the observed price. More specifically, we assume the componentwise form
\begin{eqnarray}
Z_{t_i}^{(k)}  = X_{t_i}^{(k)} + \phi(Q_{t_i}^{(k)}, \theta_0^{(k)}).
\end{eqnarray}
Accordingly, we consider the estimation of $\theta_0^{(k)}$ separately and thus we can assume that $d=1$ in what follows without loss of generality. The estimator $\widehat{\theta}^{(MSE)}$ is given by
\begin{eqnarray*}
\widehat{\theta}^{(MSE)} & = & \underset{\theta \in \Theta}{\text{argmin }} Q_N(Z,\theta), \text{ where}\\
Q_N(Z,\theta) & = & \frac{1}{2} \sum_{i=1}^N (\Delta_i Z - \mu_i (\theta))^2,
\end{eqnarray*}
where $\mu_i(\theta) = \phi(Q_{t_i}, \theta) - \phi(Q_{{t_{i-1}}}, \theta)$.


\smallskip
When $\phi$ is linear, the problem boils down to a linear regression. As a result the estimator  admits the explicit form
\begin{eqnarray}
\widehat{\theta}^{(MSE)} & = &  (\mathbb{M}^T \mathbb{M} )^{-1}\mathbb{M}^T \Delta Z,
\end{eqnarray}
where $\Delta Z := \l(\Delta_1 Z,\cdots,\Delta_N Z\r)$, and as soon as the matrix 
\beas 
\mathbb{M} := \big(\Delta_i Q^{(j)}\big)_{1 \leq i \leq N, 1 \leq j \leq l}
\eeas 
is such that $\mathbb{M}^T\mathbb{M}$ is invertible.\\

\smallskip

We now recall the limit theory associated to $\widehat{\theta}^{(MSE)}$ under the framework of \cite{li2016efficient} which in particular includes jumps with infinite activity. In the next proposition, Condition \textbf{A} assumes the local boundedness of $b$ and $\sigma$, the summability of the jump process, and several standard identifiability assumptions of most functions which depend on the parameter $\theta$ and the sequence $(Q_{t_i})_{i \in \naturels}$. Details can be found in \cite{li2016efficient}, p. 35.

\begin{proposition*}\label{thmYingYing}(Theorem 1 from \cite{li2016efficient}). Assume Condition $\textnormal{\textbf{A}}$ from \cite{li2016efficient}. Then 
$$N(\widehat{\theta}^{(MSE)} - \theta_0) = O_\proba(1).$$ 
\end{proposition*}

\section{Applications of the method}
In what follows, we state that the plug-in estimators are noise-robust for five leading examples taken from the literature, and that the central limit theorems (\ref{convergence}) and (\ref{studentizedconvergence}) hold under parametric noise. In Example \ref{volatilitysec}, we study the threshold realized volatility in the case of infinite activity jumps in price and endogeneity in arrival times. We go one step further the central limit theory of realized volatility with in \cite{li2014realized}, which includes no jumps when there is endogeneity in observation times, to allow for jumps with infinite activity. We first state the central limit theorems related to threshold realized volatility, and then the theory associated to the plug-in estimators. In Example \ref{bipowersec}, we consider the threshold bipower variation under infinite activity jumps and regular observations. In Example \ref{covariancesec}, we develop the Hayashi-Yoshida estimator of high-frequency covariance in a no-jump setup, and asynchronous and endogenous observation times. In Example \ref{functionalssec}, we consider the estimation of functionals of volatility when the price can exhibit jumps with infinite activity and observations are regular. Finally, we address the case of volatility of volatility for continuous price and volatility processes and regular observation times in Example \ref{sectionVolOfVol}.
\subsection{Threshold realized volatility}
\label{volatilitysec}
The parameter is $\xi_t = \sigma_t^2$, and the rate of convergence $\kappa = 1/2$ if observations are not contaminated by the noise. When the price is continuous and observations are regular, a popular estimator of $\Xi = \int_0^T \sigma_s^2 ds$ is RV considered in \cite{andersen2001bdistribution}, \cite{andersen2001distribution}, but also in \cite{barndorff2002econometric}, \cite{barndorff2002estimating}, \cite{meddahi2002theoretical}, etc. \cite{jacod1998asymptotic} showed that 
$$n^{1/2} \Big(RV - \int_0^T \sigma_s^2 ds \Big) \rightarrow \calm\mathcal{N} \Big(0, 2 T\int_0^T \sigma_s^4 ds \Big).$$
When observations are not regular, the AVAR is equal to $2 T\int_0^T \sigma_t^4 dH_t$, where $H_t = \lim  T^{-1}N \sum_{t_{i} \leq t} (t_i - t_{i-1})^2$ is the so-called "quadratic variation of time" (see \cite{zhang2001martingales} and \cite{mykland2006anova}), provided that such a quantity exists. When observations are endogenous, \cite{li2014realized} show that the limit distribution of $n^{1/2} (RV - \Xi)$ includes an asymptotic bias and that the related AVAR is altered. In addition, they prove that the informational content of arrival times can be useful to estimate the asymptotic bias and the AVAR.

\smallskip
Our aim is to allow for parametric noise in this endogenous setting, while also including jumps in the price process. As far as the authors know, no general theory\footnote{Remark 6 (p. 36) in \cite{li2016efficient} suggests that the threshold RV estimator can be used under endogeneity, but there is no formal proof and this is limited to the case of jumps with finite activity.} includes general endogeneity and jumps, even when observations are not noisy. Accordingly, we first extend the results of \cite{li2014realized} when adding jumps. Then, we show that the technology of this paper applies in such a general setting, and this part essentially boils down to applying the arguments of \cite{li2016efficient}.

\smallskip
Although no theory exists under endogeneity, Theorem 13.2.4 (p. 383) in \cite{jacod2011discretization} can be used when observations are regular. We consider a similar threshold RV, originally in the spirit of \cite{mancini2009non} and \cite{mancini2011speed}, and defined as $\widetilde{\Xi} = \sum_{i=1}^N \l(\Delta_i X\r)^2 \mathbf{1}_{ \{\mid \Delta_i X  \mid \leq w_i  \} }$, where $w_i = \alpha \Delta_i t^{\bar{\omega}}$, $\bar{\omega} \in (1/(2(2-r)), 1/2)$ and $\alpha > 0$ is a tuning parameter. In the next theorem, we provide the related central limit theorem and show that the condition of our paper holds.

\begin{theorem*}
\label{volatility} We assume that $\inf_{t \in (0,T]} \sigma_t > 0$. We further suppose that there exists non random $\widetilde{u}_t$ and $\widetilde{v}_t$ such that
\begin{eqnarray}
\label{volatilityass2} n \sum_{0 < t_i \leq t} \l(\Delta_i {X'}\r)^4 & \rightarrow^{\proba} & \int_0^t \widetilde{u}_s \sigma_s^4 ds,\\
\label{volatilityass1}n^{1/2} \sum_{0 < t_i \leq t} \l(\Delta_i {X'}\r)^3 & \rightarrow^{\proba} & \int_0^t \widetilde{v}_s \sigma_s^3 ds,
\end{eqnarray}
where $\widetilde{u}_t \sigma_t^4$, $\widetilde{v}_t \sigma_t^3$ and $\widetilde{v}_t^2 \sigma_t^4$ are integrable, and $\widetilde{v}_t$ locally bounded and bounded away from 0. Furthermore, we assume that $t_i$, $b_t$, $\sigma_t$ and $\delta$ are generated by finitely many Brownian motions\footnote{i.e. we assume that $t_i$ are $\mathbf{G}$-stopping times,  where $\mathbf{G} =(\calg_t)_{t \in [0,T]}$ is a sub-filtration of $\mathbf{F}$ generated by finitely many Brownian motions, and that $b_t$, $\sigma_t$ and $\delta$ are adapted to $\mathbf{G}$.}. Finally we assume that $N/n \rightarrow^\proba F$ for some random variable $F$, and that $n \Delta_i t$ are locally bounded and locally bounded away from 0. Then, stably in law as $n \rightarrow \infty$, we have 
\begin{eqnarray}
\label{thRV3} N^{1/2} ( \widetilde{\Xi} - \Xi) \rightarrow \frac{2}{3} \int_0^T v_s \sigma_s dX_s' + \int_0^T \sqrt{\frac{2}{3}u_s - \frac{4}{9} v_s^2} \sigma_s^2 dB_s,
\end{eqnarray}
where $v_s = \sqrt{F} \widetilde{v}_s$, $u_s = F \widetilde{u}_s$ and $B_t$ is a standard Brownian motion independent of the other quantities\footnote{Here and in the other theorems, we mean that $B_t$ is independent of the underlying $\sigma$-field $\mathbf{F}$.}. Moreover, we have
\begin{eqnarray}
\label{thRV2}
N^{1/2} (\widehat{\Xi} - \Xi) & \rightarrow & \frac{2}{3} \int_0^T v_s \sigma_s dX_s' + \int_0^T \sqrt{\frac{2}{3}u_s - \frac{4}{9} v_s^2} \sigma_s^2 dB_s.
\end{eqnarray}
\end{theorem*}
\begin{remark*}
\textnormal{If observations are regular, then $F=1$, $u_s = 3T$ and $v_s = 0$ for all $s \in [0,T]$. Therefore,
(\ref{thRV3}) and (\ref{thRV2}) can be specified as}
\begin{eqnarray}
n^{1/2} ( \widetilde{\Xi} - \Xi) & \rightarrow & \calm\mathcal{N} \Big(0, 2 T \int_0^T \sigma_s^4 ds \Big),\\
n^{1/2} ( \widehat{\Xi} - \Xi) & \rightarrow & \calm\mathcal{N} \Big(0, 2 T \int_0^T \sigma_s^4 ds \Big).
\end{eqnarray}
\end{remark*}
We provide now jump-robust estimators of $AB =(2/3) \int_0^T v_s \sigma_s dX_s'$ and $AVAR = \int_0^T (\frac{2}{3}u_s - \frac{4}{9} v_s^2) \sigma_s^4 ds$ based on the non jump-robust estimators provided in  \cite{li2014realized}. Accordingly, we chop the data into $B$ blocks of $h$ observations (except for the last block which might include less observations). We set $h=\lfloor n^{\beta} \rfloor$, where $1/2 < \beta < 1$. We can estimate $v_{t_{hi}}\sigma_{t_{hi}}$ as
$$\widetilde{v \sigma}_{i} =  \frac{N^{1/2} \sum_{j=h(i-1)+1}^{hi} \l(\Delta_j X\r)^3 \mathbf{1}_{ \{\mid \Delta_j X  \mid \leq w_j  \} }}{\sum_{j=h(i-1)+1}^{hi} \l(\Delta_j X\r)^2 \mathbf{1}_{ \{\mid \Delta_j X  \mid \leq w_j  \} }},$$ 
and AB and AVAR as
\begin{eqnarray*}
\widetilde{AB} & = & \sum_{i=1}^{B} \underbrace{\frac{2}{3} \widetilde{v \sigma}_{i} \Bigg\{ \sum_{j=h(i-1)+1}^{hi} \Delta_j X \mathbf{1}_{ \{\mid \Delta_j X  \mid \leq w_j  \} }  \Bigg\}}_{\widetilde{AB}_i},\\
\widetilde{AVAR} & = & \frac{2N}{3}  \sum_{i=1}^N \l(\Delta_i X\r)^4 \mathbf{1}_{ \{\mid \Delta_i X  \mid \leq w_i  \} }  - \sum_{i=1}^{B} \widetilde{AB}_{i}^2.
\end{eqnarray*}
Recalling that $\widehat{AB}$ and $\widehat{AVAR}$ are constructed respectively as $\widetilde{AB}$ and $\widetilde{AVAR}$ when replacing $X$ by $\widehat{X}$, we provide now the studentized version of the previous central limit theorems. 
\begin{corollary*}
\label{volatilitystudent}
We have 
\begin{eqnarray}
\label{volatilitystudent1} N^{1/2} \frac{ \widetilde{\Xi} - N^{-1/2} \widetilde{AB} - \Xi}{\sqrt{\widetilde{AVAR}}} & \rightarrow & \mathcal{N} (0,1),\\
\label{volatilitystudent2} N^{1/2} \frac{ \widehat{\Xi} - N^{-1/2} \widehat{AB} - \Xi}{\sqrt{\widehat{AVAR}}} & \rightarrow & \mathcal{N} (0,1).
\end{eqnarray}
\end{corollary*}
\begin{remark*}\textnormal{ If observations are regular, there is no asymptotic bias and $AVAR$ can be estimated using the plug-in estimator of quarticity obtained in Section \ref{functionalssec}. In view of Theorem \ref{theoremG} which implies the consistency of the plug-in estimator, we obtain directly by the stable convergence obtained in Theorem \ref{volatility} that (\ref{volatilitystudent2}) holds.}
\end{remark*}

\begin{remark*} \textnormal{(estimating volatility under i.i.d noise)
Alternative approaches to estimate integrated volatility under latent i.i.d noise include and are not limited to: the Quasi-Maximum Likelihood Estimator (QMLE) from \cite{ait2005often}  which was later shown to be robust to time-varying volatility in \cite{xiu2010quasi}, the Two-Scale Realized Volatility in \cite{zhang2005tale}, the multi-Scale realized volatility in \cite{zhang2006efficient}, the pre-averaging approach  in \cite{jacod2009microstructure}, realized kernels in \cite{barndorff2008designing} and the spectral approach considered in \cite{altmeyer2015functional} based on \cite{reiss2011asymptotic}. \cite{clinet2018efficient} discussed AVAR reduction when considering local estimators. In addition, \cite{li2013volatility} consider endogenous arrival times.}
\end{remark*}

\subsection{Threshold bipower variation}
\label{bipowersec}
Here again $\xi_t = \sigma_t^2$. The bipower variation $BV = \frac{\pi}{2} \sum_{i=2}^N \mid\Delta_i X \mid \mid \Delta_{i-1} X \mid$ (more generally  multipower variation from \cite{barndorff2004power} and \cite{barndorff2006econometrics}) was originally introduced as an alternative measure robust to finite-activity jumps. In case of regular observations and no jump, \cite{barndorff2006central} and \cite{barndorff2006limit} established the central limit theory. See also \cite{kinnebrock2008note} for related development. In case of finite-activity jumps, see also \cite{barndorff2006limit2}.

\smallskip 
If jumps exhibit infinite activity,  \cite{vetter2010limit} shows that BV is no longer consistent, but the jump-robust threshold estimator 
$$\widetilde{\Xi} = \frac{\pi}{2} \sum_{i=2}^N \mid \Delta_i X \mid \mathbf{1}_{ \{\mid \Delta_i X  \mid \leq w  \} } \mid \Delta_{i-1} X \mid \mathbf{1}_{ \{\mid \Delta_{i-1} X  \mid \leq w  \} }$$
is consistent, where  $w = \alpha \Delta^{\bar{\omega}}$, $\bar{\omega} \in (0, 1/2)$. Moreover, he also shows the related central limit theory. See also \cite{corsi2010threshold} for related work. Finally, the general theory (Theorem 13.2.1 (p. 380)) from \cite{jacod2011discretization} can be applied too. All those papers have in common that they assume regular observations, and we follow the same setting to show that the techniques of this paper can be used in this example too. We provide the formal result in what follows.

\begin{theorem*} \label{bipower}
We have that 
\begin{eqnarray}
\label{bipowereq0}
n^{1/2} \big(\widehat{\Xi} - \widetilde{\Xi} \big) \rightarrow^\proba 0.
\end{eqnarray}
 In particular,  
 stably in law as $n \rightarrow \infty$, 
 \begin{eqnarray}
 \label{bipowereq}
 n^{1/2} (\widehat{\Xi} - \Xi ) \rightarrow \frac{\pi}{2}\sqrt{\l(1+\frac{4}{\pi} - \frac{12}{\pi^2}\r)T} \int_0^T \sigma_s^2 d B_s ,
 \end{eqnarray}
 where $B_t$ is a Brownian motion independent of the other quantities.
\end{theorem*}
In this example, we have that $AVAR = \frac{\pi^2}{4}(1+\frac{4}{\pi} - \frac{12}{\pi^2})T \int_0^T \sigma_s^4 ds$, which can be estimated by $\widehat{AVAR} = \frac{\pi^2}{4}(1+\frac{4}{\pi} - \frac{12}{\pi^2})T \widehat{\int_0^T \sigma_s^4 ds}$, where the plug-in estimator of quarticity $\widehat{\int_0^T \sigma_s^4 ds}$ is defined as a particular case of Section \ref{functionalssec} (i.e $\widehat{\int_0^T \sigma_s^4 ds}$ corresponds to the estimator given in (\ref{eqGEstimator}) below with $g(x) = x^2$). We also provide the related studentized central limit theorem.
\begin{corollary*}
\label{bipowerstudent}
We have 
\begin{eqnarray}
\label{bipowerstudent0} n^{1/2} \frac{ \widehat{\Xi} - \Xi}{\sqrt{\widehat{AVAR}}} & \rightarrow & \mathcal{N} (0,1).
\end{eqnarray}
\end{corollary*}

\subsection{Hayashi-Yoshida estimator of high-frequency covariance}
\label{covariancesec}
We assume here that $X_t$ is 2-dimensional and that $\xi_t = \rho_t \sigma_t^{(1)} \sigma_t^{(2)}$, where the high-frequency correlation $\rho_t$ satisfies $d \langle W^{(1)}, W^{(2)} \rangle_t = \rho_t dt$. The rate of convergence is $\kappa = 1/2$ in this problem too. We consider that observations are non-synchronous. In this framework and assuming that the price is continuous, \cite{hayashi2005covariance} bring forward the so-called Hayashi-Yoshida estimator and establish the consistency in case sampling times are independent from the price process. This is extended in an endogenous setting in \cite{hayashi2008consistent}. The related central limit theory can be found in \cite{hayashi2008asymptotic}, \cite{hayashi2011nonsynchronous} and \cite{potiron2017estimation}, where the latter work considers general endogenous arrival times. See also the remarkable work from \cite{bibinger2015estimating} and \cite{martin2019null} in a jumpy setting, and \cite{koike2014estimator}, \cite{koike2014limit} and \cite{koike2016estimation} which incorporates jumps, noise and some kind of endogeneity into the model.

\smallskip
As we want to allow for quite exotic endogenous models, we follow \cite{potiron2017estimation}. In particular, we assume no jumps in the setup. We describe the hitting boundary with time process (HBT) model introduced in the subsequent paper in what follows. In that model, eight stochastic processes (four of which are actually families of stochastic processes) are of interest, four for each asset. For the index $k=1,2$, we have the price process - $X_t^{(k)}$ - and three other stochastic processes (two of which are actually families of processes) - $Y_t^{(k)}$, $d_t^{(k)}(s)$ and $u_t^{(k)}(s)$ - related to the observation times of that process. Those four stochastic processes can be correlated, and we further assume that $(X_t, Y_t)$ is a 4-dimensional It\^{o}-process. For the process $k=1,2$,  $Y_t^{(k)}$ stands for the continuous observation time process which drives the observation times related to $X_t^{(k)}$.  The others four processes are the down processes $d_t^{(k)}(s)$ and the up processes $u_t^{(k)}(s)$.  We assume 
that the down process takes only negative values and that the up 
process takes only positive values. A new observation time will be generated whenever one of those two processes is hit by the increment of the observation time process. 
Then, the increment of the observation time process will be reset to $0$, and the next observation time will be produced whenever the up or the down process is hit again. Formally, if we let $\alpha > 0$ stand for the tick size, we define the first observation time as $t_0^{(k)} := 0$ and recursively $t_i^{(k)}$ as
\begin{eqnarray}
 \label{generateobstimes} t_i^{(k)} := \inf \Big\{ t > t_{i-1}^{(k)} : \Delta Y_{[t_{i-1}^{(k)}, t]}^{(k)} \notin \big[ \alpha d_{t}^{(k)} 
 \big( t - t_{i-1}^{(k)} \big), \alpha u_{t}^{(k)} \big(t - t_{i-1}^{(k)} \big) 
\big] \Big\},
\end{eqnarray}
where $\Delta Y_{[a,b]}^{(k)} := Y_b^{(k)} - Y_a^{(k)}$. We define the Hayashi-Yoshida estimator as
\begin{eqnarray}
\label{HY0} \widetilde{\Xi} := \sum_{0 < t_{i}^{(1)}\text{ , } t_{j}^{(2)} < T} 
\Delta_i X^{(1)} \Delta_j X^{(2)} 
\mathbf{1}_{ \big\{ [ t_{i-1}^{(1)}, t_{i}^{(1)} ) \cap [ t_{j-1}^{(2)}, 
t_{j}^{(2)} ) \neq \emptyset \big\} }.
\end{eqnarray}
In the asymptotic theory, we let $\alpha \rightarrow 0$. For the sake of Remark 5 (p. 25) in \cite{potiron2017estimation}, $\alpha^{-1}$ is of the same order as $n^{1/2}$. We can now show that the techniques of this paper hold in this case too.
\begin{theorem*} \label{covariance}
As the tick size $\alpha \rightarrow 0$, we have that 
\begin{eqnarray}
\label{proof201710301135} \alpha^{-1} \big(\widehat{\Xi} - \widetilde{\Xi} \big) \rightarrow^\proba 0.
\end{eqnarray}
 In particular, under the assumptions of \cite{potiron2017estimation}, there exist $AB$ and a process $AV_t$ such that 
 stably in law as the tick size $\alpha \rightarrow 0$, 
 \begin{eqnarray}
 \label{theorem}
 \alpha^{-1} (\widehat{\Xi} - \Xi ) \rightarrow AB + \int_0^T \left( 
 AV_s \right)^{1/2} d B_s ,
 \end{eqnarray}
 where $B_t$ is a Brownian motion independent of the other quantities, $AB$ and $AV_t$ are defined in Section 4.3 of \cite{potiron2017estimation}.
\end{theorem*}
We define $\widetilde{AB}$ and $\widetilde{AVAR}$ following respectively (46) and (47) in \cite{potiron2017estimation} (Section 5, p. 28). Note that $\widetilde{AB}$ and $\widetilde{AVAR}$ are already of the right asymptotic order in the sense that $\alpha^{-1} \widetilde{AB} \rightarrow^\proba AB$ and $\alpha^{-2} \widetilde{AVAR} \rightarrow^\proba \int_0^T  
 AV_s  ds$ (see (48) and (49) in Corollary 4 of the cited paper). We provide now the studentized version of (\ref{theorem}).

\begin{corollary*}
\label{covariancestudent}
We have 
\begin{eqnarray}
\label{covariancestudent0} \frac{ \widehat{\Xi} - \widehat{AB} - \Xi}{\sqrt{\widehat{AVAR}}} & \rightarrow & \mathcal{N} (0,1).
\end{eqnarray}
\end{corollary*}

\subsection{Functionals of volatility local estimator}
\label{functionalssec}
The spot parameter is $\xi_t = g(c_t)$ for a given smooth function $g$ on $\calm_d^+$, the set of all non-negative symmetric $d\times d$ matrices.
The problem was initiated by \cite{barndorff2002econometric}. See also \cite{barndorff2006central}, \cite{mykland2012econometrics} (Proposition 2.17, p. 138) and \cite{renault2017efficient} for related developments. Here, the rate of convergence is $\kappa = 1/2$ again.

\smallskip
Local estimation (\cite{mykland2009inference}, Section 4.1, p. 1421-1426) can make
the mentioned estimators efficient. \cite{jacod2013quarticity} extended the method in several ways. To do that, they first propose an estimator of the spot volatility  $\widetilde{c}_i$, and then take a Riemann sum of $g(\widetilde{c}_i)$. 

\smallskip
For any matrix $a \in \calm_d^+$, the related $a^{ij}$ stands for the $(i,j)$-component of $a$. Moreover, for $b \in \reels$, $[b]$ stands for the floor of $b$. Several results are of interest in \cite{jacod2013quarticity}. In its most useful form (from our point of view), the estimator takes on the form
\bea 
\widetilde{\Xi} = \Delta \sum_{i=1}^{[T/\Delta]-k+1}\l\{g(\widetilde{c}_{i}) - \inv{2k}\sum_{j,q,l,m=1}^d \partial_{jq,lm}^2g(\widetilde{c}_{i})\l(\widetilde{c}_{i}^{jl}\widetilde{c}_{i}^{qm}+\widetilde{c}_{i}^{jm}\widetilde{c}_{i}^{ql}\r)\r\},
\label{eqGEstimator}
\eea 
with 
$$\widetilde{c}_{i}^{lm}=\inv{k \Delta} \sum_{j=0}^{k-1}\Delta_{i+j} X^l\Delta_{i+j} X^m \mathbf{1}_{\{\|\Delta_{i+j} X\| \leq w\}}, $$
for two sequences of integers $k$ and $w = \alpha \Delta^{\bar{\omega}}$ for some $\alpha >0$, and 
$$\frac{2p-1}{2(2p-r)}\leq\bar{\omega}<\half,$$ 
where we suppose that 
\bea 
\|\partial^jg(x)\| \leq K(1+\|x\|^{p-j}),\textnormal{  } j=0,1,2,3
\label{momentG}
\eea 
for some constants $p \geq 3$, $K>0$.
In Equation (\ref{eqGEstimator}), $\widetilde{c}_{i}$ corresponds to an estimator of the spot volatility matrix, the first term is part of the Riemann sum, while the second term is required to remove the asymptotic bias of the first term in $\widetilde{\Xi}$, which explodes asymptotically. We show that the associated plug-in estimator $\widehat{\Xi}$ enjoys the same limit theory as $\widetilde{\Xi}$. More precisely, we have the following result.

\begin{theorem*}\label{theoremG}
Assume that $k^2\Delta \to0$, $k^3\Delta \to\infty$. Let $\widetilde{\Xi}'$ be the estimator defined as in (\ref{eqGEstimator}) where $X_t$ is replaced by its continuous part $X_t'$. Then, we have the convergence
\begin{eqnarray}
n^{1/2}\l(\widehat{\Xi} - \widetilde{\Xi}'\r) \to^\proba 0.
\end{eqnarray}
Moreover, stably in law, we have the convergence
\begin{eqnarray}
 n^{1/2}\l(\widehat{\Xi} -  \Xi \r) \rightarrow \int_0^T{\sqrt{T\overline{h}(c_s)}dB_s},
\end{eqnarray}
where for $x \in \calm_d^+$,
$$ \overline{h}(x) = \sum_{j,q,l,m=1}^d \partial_{jq}g(x)\partial_{lm}g(x)(x^{jl}x^{qm}+x^{jm}x^{ql}),$$
and where $B$ is a standard Brownian motion independent of the other quantities. 
\end{theorem*}
In particular, note that the asymptotic variance in the stable convergence can be expressed as
\beas 
AVAR = T\int_0^T{\overline{h}(c_s)ds},
\eeas 
so that we naturally define the asymptotic variance estimator as 
\beas 
\widehat{AVAR} =  T \Delta \sum_{i=1}^{[t/\Delta]-k+1} \overline{h}(\widehat{c}_{i}).
\eeas  

We easily deduce from Corollary 3.7 p. 1471 in \cite{jacod2013quarticity} the following studentized version of the above central limit theorem. 
\begin{corollary*}\label{corolStudentG}
Under the assumptions of the previous theorem, we have the stable convergence in law
\beas 
\frac{n^{1/2}\big(\widehat{\Xi} -  \Xi \big)}{\sqrt{\widehat{AVAR}}} \to \caln(0,1).
\eeas 
\end{corollary*}

\begin{remark*} \textnormal{(estimation of functionals of volatility under i.i.d noise)
Under i.i.d noise, no result with a general function $g(c_t)$ is available. Alternative approaches include: \cite{jacod2010limit} for even power, \cite{mancino2012estimation} and also \cite{andersen2014robust} in the special case of quarticity, and also \cite{altmeyer2015functional} when considering the tricity. See also the work from \cite{potiron2016local} (Section 4.2) for a local maximum-likelihood estimation with noise variance vanishing asymptotically. }
\end{remark*}

\subsection{Volatility of volatility}\label{sectionVolOfVol}
In this section we assume that $X_t$ is 1-dimensional and we are interested in the spot parameter $\xi_t = \widetilde{\sigma}_t^2$ which corresponds to the so-called volatility of volatility process defined in (\ref{volatilityDefinition}). As far as we know, there is no result in the literature including noise into the model, but in the non-noisy scenario one can consult \cite{vetter2015estimation} (Theorem 2.5 and Theorem 2.6) and \cite{mykland2012efficient} (Theorem 7 and Corollary 2). We follow here the former author, and aim to show the robustness of Theorem 2.6 when using plug-in estimators. Accordingly, we hereafter assume that both $X_t$ and $c_t$ are continuous processes, i.e. $\delta = \widetilde{\delta} = 0$ in (\ref{efficientPrice})-(\ref{volatilityDefinition}). To our knowledge, the case with jumps in $X_t$ and/or $c_t$ remains an open question. The rate of convergence is $\kappa = 1/4$. Introducing the spot volatility estimator\footnote{Note that the definition of $\widetilde{c}_{i}$ slightly diverges from the previous section.} for $i \in \{0,\cdots,n-k \}$,
\beas
\widetilde{c}_{i} := \frac{n}{k}\sum_{j=1}^{k}\l(\Delta_{i+j}X\r)^2, 
\eeas
and the spot quarticity estimator
\beas 
\widetilde{q}_i := \frac{n^2}{3k}\sum_{j=1}^{k}\l(\Delta_{i+j}X\r)^4,
\eeas 
the author defines the volatility of volatility estimator (see (2.5) on p. 2399 in the cited work) as 
\beas
\widetilde{\Xi} := \sum_{i=0}^{[t/\Delta] - 2k}\l\{\frac{3}{2k}\l(\widetilde{c}_{i+k} - \widetilde{c}_{i}\r)^2 - \frac{6}{k^2}\widetilde{q}_i\r\}. 
\eeas 
Letting $\widehat{c}_i$, $\widehat{q}_i$, and $\widehat{\Xi}$ be the corresponding plug-in estimators, we obtain the following results.

\begin{theorem*} \label{thmVolofVol}
Assume that $k = cn^{1/2}+o(n^{1/4})$ for some $c>0$. Then stably in law,
\beas
\sqrt{\frac{n}{k}}\l( \widehat{\Xi} - \Xi \r) \to \sqrt{T}\int_0^T \alpha_sdB_s,
\eeas
where $B_t$ is a Brownian motion independent from the other quantities and
\beas 
\alpha_s^2 = \frac{48}{c^4}\sigma_s^8 + \frac{12}{c^2}\sigma_s^4\widetilde{\sigma}_s^2 + \frac{151}{70}\widetilde{\sigma}_s^4.
\eeas 
\end{theorem*}
Moreover, if we define
\begin{eqnarray*}
G^{(1)} & = & \frac{T}{n}\sum_{i=0}^{[t/\Delta] - k} \widehat{q}_i^2,\\
G^{(2)} & = & T\sum_{i=0}^{[t/\Delta] - 2k}\l\{\frac{3}{2k}\l(\widehat{c}_{i+k} - \widehat{c}_{i}\r)^2 - \frac{6}{k^2}\widehat{q}_i\r\}\widehat{q}_i,\\
G^{(3)} & = & \frac{Tn}{k^2}\sum_{i=0}^{[t/\Delta] - 2k}\l(\widehat{c}_{i+k} - \widehat{c}_{i}\r)^4,
\end{eqnarray*}
and finally
\beas 
\widehat{AVAR} = \frac{453}{280}G^{(3)} - \frac{n}{k^2}\frac{486}{35}G^{(2)}-\frac{n^2}{k^4}\frac{1038}{35}G^{(1)}, 
\eeas 
we can derive the following studentized version of the previous central limit theorem.
\begin{corollary*}\label{corVolofVol}
Under the assumptions of the previous theorem, we have the stable convergence in law, when $k$ has the optimal rate  $c \sqrt{n}$ for $c>0$
\beas 
n^{1/4}\frac{  \widehat{\Xi} -  \Xi  }{\sqrt{c\widehat{AVAR}}} \to \caln(0,1).
\eeas 
\end{corollary*}
\section{Conclusion}
This paper develops plug-in estimators to estimate high-frequency quantities under parametric noise on five different examples. We do not find any particular difficulty when working out the theory of those examples. Another example of application can be found in \cite{andersen2019time}.

\section*{Acknowledgments}

We would like to thank Selma Chaker, Yingying Li, Mathias Vetter, Xinghua Zheng, 
Manh Cuong Pham, two anonymous referees, the participants of the 2nd International Conference on Econometrics and Statistics in Hong Kong, and the Econometric Society Australasian Meeting 2018 in Auckland for helpful discussions and advice. The research of Yoann Potiron is supported by a special private grant from Keio University and Japanese Society for the Promotion of Science Grant-in-Aid for Young Scientists No. 60781119. The research of Simon Clinet is supported by Japanese Society for the Promotion of Science Grant-in-Aid for Young Scientists No. 19K13671.


\section{Proofs}

\subsection{Preliminaries}
Due to our assumptions of local boundedness on $b_t$, $\widetilde{b}_t$, $c_t$ and $\widetilde{c}_t$, (\ref{infoass}) and (\ref{thetahat}), it is sufficient (see, e.g., Lemma 4.4.9 along with Proposition 2.2.1 in \cite{jacod2011discretization}) to assume throughout the proofs the following stronger assumption.\\
\textbf { (H) } We have that $b_t$, $\widetilde{b}_t$, $c_t$ and $\widetilde{c}_t$ are bounded. Moreover, there exists $K >0$ such that $\|\widehat{\theta} -\theta_0\| \leq K/n$, and $\max_{i,j,k} \l|Q_{t_i^{(k)}}^{(k,j)}\r| \leq K$.\\
Since the last two properties on $\widehat{\theta}$ and $Q$ are not directly implied by Proposition 2.2.1 from \cite{jacod2011discretization}, we now detail a general localization procedure in the next proposition, which we apply to the above particular cases in Corollary \ref{corolLocalization}. In the next lemma, if $A$ is a random event, $\overline{A} $ stands for $ \Omega - A$.   
\begin{proposition*} (Localization)
Let $(A_n^K)_{n \in \naturels, K \in \reels_+}$ be a doubly-indexed family of events such that $\lim_{K \to +\infty} \sup_{n \in \naturels} \proba[\overline{A_n^K}] = 0$. Let $(X_n)_{n \in \naturels}$ be a sequence of $\reels^d$-valued random variables for some $d \geq 1$, and $X$ another $\reels^d$-valued random variable, and assume that either of the following properties hold.
\begin{enumerate}
    \item (Local convergence in probability) For any $K\geq0$, $(X_n - X)\mathbb{1}_{ A_n^K} \to^\proba 0$.
    \item (Local convergence in distribution) For any $K\geq0$, for any $f$ continuous and bounded, $\esp[f(X_n)\mathbb{1}_{ A_n^K} ] \to \esp[f(X)]$.
\end{enumerate}
Then we have respectively
\begin{enumerate}
    \item $X_n \to^\proba X$.
    \item  $X_n \to^d X$.
\end{enumerate}
\end{proposition*}

\begin{proof}
We prove the convergence in probability first. Fix $\epsilon >0$ and $\eta >0$, and note that 
\beas \proba[|X_n -X| \geq \eta] &\leq& \proba\l[|X_n-X|\mathbb{1}_{ A_n^K} \geq \frac{\eta}{2}\r] +\proba\l[|X_n-X|\mathbb{1}_{ \overline{A_n^K}} \geq \frac{\eta}{2}\r]\\
&\leq & \proba\l[|X_n-X|\mathbb{1}_{ A_n^K} \geq \frac{\eta}{2}\r] +\proba\l[\overline{A_n^K}\r]. 
\eeas 
By taking $K$ large enough, we can assume that the second term in the right-hand side is dominated by $\epsilon$. Next, by taking $n$ large enough, we may assume the first term to be smaller than $\epsilon$ as well by the local convergence in probability. This proves $X_n \to^\proba X$. Next we prove the convergence in distribution. We have
\beas 
|\esp[f(X_n)] - \esp[f(X)]| &=& |\esp[f(X_n)\mathbb{1}_{ A_n^K}] - \esp[f(X)] +\esp[f(X_n)\mathbb{1}_{ \overline{A_n^K}}]|\\
&\leq& |\esp[f(X_n)\mathbb{1}_{ A_n^K}] - \esp[f(X)]| +C\proba \l[ \overline{A_n^K}\r]
\eeas 
for some constant $C$ using the boundedness of $f$. Again, taking $K$ large enough makes the third term arbitrary small, and then taking $n \to +\infty$ makes the difference between the first two terms tend to $0$, wich proves $X_n \to^d X$.  
\end{proof}

\begin{corollary*} \label{corolLocalization}
When proving the consistency of the estimator $\widehat{\Xi}$ toward $\Xi$ or the asymptotic normality $n^{\kappa}(\widehat{\Xi} - \Xi) \to \calm\caln(AB,AVAR)$, we may assume that there exists $K>0$ (which may be arbitrary large) such that $\|\widehat{\theta} -\theta_0\| \leq K/n$, and $\max_{i,j,k} \l|Q_{t_i^{(k)}}^{(k,j)}\r| \leq K$.  
\end{corollary*}

\begin{proof}
We show the case $\|\widehat{\theta} -\theta_0\| \leq K/n$, the case $\max_{i,j,k} |Q_{t_i^{(k)}}^{(k,j)}| \leq K$ being the same. For the consistency, we apply the previous proposition with $X_n = \widehat{\Xi}$, $X= \Xi$, and $A_n^K=\l\{ \|\widehat{\theta} -\theta_0\| \leq K/n\r\}$. By hypothesis (\ref{ratethetahat}), $n(\widehat{\theta} -\theta_0)$ is stochastically bounded which exactly means that $\lim_{K \to +\infty} \sup_{n \in \naturels} \proba[ n\|\widehat{\theta} -\theta_0\| \geq K] = 0$ (recall that $\widehat{\theta}$ depends on $n$). For the central limit theory, apply the local convergence distribution with $X_n = n^\kappa (\widehat{\Xi} -\Xi)$, $X \sim  \calm\caln(AB,AVAR)$, and again $A_n^K=\l\{ \|\widehat{\theta} -\theta_0\| \leq K/n\r\}$.
\end{proof}

All along the proofs, $C$ is a constant that may vary from one line to the next. 
We further provide some notation related to the decomposition  (\ref{efficientPrice}) of the efficient price, i.e. that
\begin{eqnarray}
\nonumber X_t & = & X_0 + \int_0^t b_s ds + \int_0^t \sigma_s dW_s + \int_0^t \int_\reels \delta (s,z) \mathbf{1}_{\{\mid \mid \delta (s,z) \mid \mid \leq 1\}} (\mu - \nu) (ds, dz)\\
 \nonumber & & + \int_0^t \int_\reels \delta (s,z) \mathbf{1}_{\{\mid \mid \delta (s,z) \mid \mid > 1\}} \mu (ds, dz),\\
\label{decompoX} & := & X_0 + B_t + M_t^{c} + M_t^d + J_t^{b}.
\end{eqnarray}
Note that in this decomposition $M_t^c$ (resp. $M_t^d$) is a continuous (resp. purely discontinuous) local martingale (see the discussion in Section 2.1.2 in \cite{jacod2011discretization}). Finally, we introduce $\Delta_i X(\theta) := \Delta_i X + \psi_i(\theta)$ where $\psi_i(\theta) := \mu_i(\theta_0) - \mu_i(\theta)$. In particular, note that  $\Delta_i \widehat{X} = \Delta_iX(\widehat{\theta})$. Similarly we define $\Delta_i X'(\theta) := \Delta_i X' + \psi_i(\theta)$ and $\Delta_i \widehat{X}' = \Delta_i X'(\widehat{\theta})$, corresponding to the estimated increments when the jump part $J$ has been removed. Moreover, $\esp_s$ is defined as the conditional expectation given $\mathcal{F}_s$.
\subsection{Proof of Theorem \ref{volatility}}
For this proof, due to our assumptions in Theorem \ref{volatility} and using the same argument as for Assumption \textbf{(H)} we further make the following assumption.\\ 
\textbf{ (H') } We have that $n \Delta_i t$ and $\widetilde{v}_t$ are bounded and bounded away from 0.\\
Note that (\ref{thRV3}) is a particular case of (\ref{thRV2}) when $\phi = 0$. In what follows, we directly prove the general case (\ref{thRV2}). First of all, as $N/n \rightarrow^{\proba} F$, it is sufficient to show the stable convergence in law 
\begin{eqnarray}
\label{thRV30} n^{1/2} ( \widetilde{\Xi} - \Xi) \rightarrow \frac{2}{3} \int_0^T \widetilde{v}_s \sigma_s dX_s' + \int_0^T \sqrt{\frac{2}{3}\widetilde{u}_s - \frac{4}{9} \widetilde{v}_s^2} \sigma_s^2 dB_s.
\end{eqnarray}
Second, note that if we can prove that
\begin{eqnarray}
\label{proof1} n^{1/2} \sum_{i=1}^N \l(\Delta_i \widehat{X}\r)^2 \mathbf{1}_{ \{\mid \Delta_i \widehat{X}  \mid \leq w_i  \} } = n^{1/2} \sum_{i=1}^N \l(\Delta_i X^{'}\r)^2 + o_\proba (1),
\end{eqnarray}
then (\ref{thRV3}) holds in view of Theorem 1 (p. 585) in \cite{li2014realized} together with  the assumptions of Theorem \ref{volatility}. Accordingly, we show (\ref{proof1}) in what follows. On the account of the decomposition (\ref{decompoX0}), we have
\begin{eqnarray*}
n^{1/2} \sum_{i=1}^N \l(\Delta_i \widehat{X}\r)^2 \mathbf{1}_{ \{\mid \Delta_i \widehat{X}  \mid \leq w_i  \} } &=& n^{1/2} \sum_{i=1}^N \l(\Delta_i \widehat{X}'\r)^2 \mathbf{1}_{ \{\mid \Delta_i \widehat{X}  \mid \leq w_i  \} } +2n^{1/2} \sum_{i=1}^N \Delta_i \widehat{X}'\Delta_i J  \mathbf{1}_{ \{\mid \Delta_i \widehat{X}  \mid \leq w_i  \} } \\
& &+  n^{1/2} \sum_{i=1}^N \Delta_i J^2 \mathbf{1}_{ \{\mid \Delta_i \widehat{X}  \mid \leq w_i  \} },\\
& := & I + II + III.
\end{eqnarray*}
We will show in what follows that $I = n^{1/2} \sum_{i=1}^N \l(\Delta_i X^{'}\r)^2 + o_\proba (1)$, $II = o_\proba (1)$,  and $III = o_\proba(1)$.

\smallskip
We start with $I$. By definition, we have 
\begin{eqnarray*}
I & = &  n^{1/2} \sum_{i=1}^N \l(\Delta_i \widehat{X}'\r)^2 - n^{1/2} \sum_{i=1}^N \l(\Delta_i \widehat{X}'\r)^2 \mathbf{1}_{ \{\mid \Delta_i \widehat{X}  \mid > w_i  \} }.
\end{eqnarray*}
We show now that $n^{1/2} \sum_{i=1}^N \l(\Delta_i \widehat{X}'\r)^2 \mathbf{1}_{ \{\mid \Delta_i \widehat{X}  \mid > w_i  \} } = o_\proba (1)$. We have that 
\begin{eqnarray*}
n^{1/2} \sum_{i=1}^N \l(\Delta_i \widehat{X}'\r)^2 \mathbf{1}_{ \{\mid \Delta_i \widehat{X} \mid > w_i \} } & \leq & n^{1/2} \sum_{i=1}^N \l(\Delta_i \widehat{X}'\r)^2 \mathbf{1}_{ \{\mid \Delta_i \widehat{X}'  \mid > w_i/2  \} } + n^{1/2} \sum_{i=1}^N \l(\Delta_i \widehat{X}'\r)^2 \mathbf{1}_{ \{\mid \Delta_i J  \mid > w_i/2 \} }\\
& := & A + B.
\end{eqnarray*}

We first deal with $A$. By the domination $\mathbf{1}_{ \{\mid \Delta_i \widehat{X}'  \mid > w_i/2  \} } \leq 2^k \mid \Delta_i \widehat{X}'\mid^k w_i^{-k}$, we have for any $k > 0$:
\begin{eqnarray}\label{devXprime}
|A| & \leq & Cn^{1/2} \sum_{i=1}^N w_i^{-k}|\Delta_i \widehat{X}'|^{2+k}.
\end{eqnarray}
Now, note that by Assumption \textbf{(H)} along with the fact that $\psi_i$ is $C^3$ in $\theta$ and that $\Theta$ is a compact set, we easily obtain that for any $k \geq 1$, $|\psi_i(\widehat{\theta})|^k \leq Cn^{-k}$. From here, by Assumption \textbf{ (H') } we deduce by Burkh\"{o}lder-Davis-Gundy inequality that 
\bea 
\esp |\Delta_i \widehat{X}'|^k \leq C(n^{-k/2} + n^{-k}) \leq Cn^{-k/2},\label{devContinuousTheta}
\eea  
and so we can conclude that taking $k$ large enough, $A = o_\proba(1)$ as a result of the boundedness of $n \Delta_i t$, and $N/n \rightarrow F$. 

\smallskip
Now, we deal with $B$. Remark that by \textbf{(H')} and  H\"{o}lder's inequality we have
\beas 
 |B| &\leq& 2n^{1/2}  \sum_{i=1}^N \l(\Delta_i \widehat{X}'\r)^2 |\Delta_i J| |w_i|^{-1}  \\
 &\leq & Cn^{1/2 + \bar{\omega}} \sum_{i=1}^N \l(\Delta_i \widehat{X}'\r)^2  |\Delta_i J| \\
 &\leq & Cn^{1/2 + \bar{\omega}} \l(\sum_{i=1}^N \l(\Delta_i \widehat{X}'\r)^{2p}\r)^{1/p} \l(\sum_{i=1}^N   |\Delta_i J|^q\r)^{1/q}
\eeas 
where $1/p + 1/q = 1$ and $p,q >1$. By (\ref{devXprime}) we get $\l(\sum_{i=1}^N \l(\Delta_i \widehat{X}'\r)^{2p}\r)^{1/p} = O_\proba(n^{1/p - 1})$ and since $q > 1$, we also have $\sum_{i=1}^N   |\Delta_i J|^q = O_\proba(1)$ because the jumps are summable. Indeed, note first that by application of Theorem 3.3.1, Case A, p.70 from \cite{jacod2011discretization} under assumption (A-c), with $f(x) = |x|^q = o(x)$ for $x \to 0$ since $q >1$, we have the convergence $\sum_{i=1}^N |\Delta_i J|^q \to^\proba \sum_{0 < s \leq T} |\Delta_s J|^q$. The stochastic boundedness of the left-hand side will therefore be proved if we show that the limit is finite almost surely. We can write 
\beas 
\sum_{0 < s \leq T} |\Delta_s J|^q &=&    \sum_{0 < s \leq T} |\Delta_s J|^q \mathbf{1}_{\{|\Delta_s J| \geq 1\}} + \sum_{0<s \leq T} |\Delta_s J|^q \mathbf{1}_{\{|\Delta_s J| < 1\}}.
\eeas 
The first term of the right-hand side is clearly finite since there is only a finite number of jumps larger than $1$ on the interval $[0,T]$. Moreover, for the second term, using that $|x|^q < |x|$ for $x \in [0,1)$ when $q > 1$, and using that the jumps are summable yields
\begin{eqnarray*}
\sum_{0<s \leq T} |\Delta_s J|^q \mathbf{1}_{\{|\Delta_s J| < 1\}} \leq \sum_{0<s \leq T} |\Delta_s J| < +\infty \textnormal{ a.s. }
\end{eqnarray*}
Overall this yields $B = O_\proba(n^{ 1/p+ \bar{\omega} -1/2})$, which tends to $0$ as soon as $p$ is taken larger than $(1/2-\bar{\omega})^{-1}$, which is possible since $\bar{\omega} < 1/2$. Now we conclude for $I$ by showing that we have 
\bea 
n^{1/2} \sum_{i=1}^N \l(\Delta_i \widehat{X}'\r)^2 = n^{1/2} \sum_{i=1}^N \l(\Delta_i X^{'}\r)^2 + o_\proba(1).
\label{deviationRVTheta}
\eea 
Note that
\beas
n^{1/2}\sum_{i=1}^N \l( \l(\Delta_i \widehat{X}'\r)^2 - \l(\Delta_i X^{'}\r)^2 \r) = 2n^{1/2} \sum_{i=1}^N  \Delta_i X^{'} \psi_i(\widehat{\theta}) + n^{1/2} \sum_{i=1}^N \psi_i(\widehat{\theta})^2,
\eeas 
and the second term in the right-hand side of the equation is negligible as a direct consequence of the domination $|\psi_i(\widehat{\theta})| \leq C/n$. We show now that the first term is also negligible. By the mean value theorem, we also have for some $\overline{\theta} \in [\theta_0,\widehat{\theta}]$ that  
\begin{eqnarray}  
\label{mvtproof}
n^{1/2}\sum_{i=1}^N  \Delta_i X^{'} \psi_i(\widehat{\theta}) = n^{1/2}(\widehat{\theta} -\theta_0)^T\sum_{i=1}^N  \Delta_i X^{'} \partial_\theta \psi_i(\theta_0) + \frac{ n^{1/2}(\widehat{\theta} -\theta_0)^T}{2}\sum_{i=1}^N  \Delta_i X^{'} \partial_\theta^2 \psi_i(\overline{\theta})(\widehat{\theta} -\theta_0).
\end{eqnarray}
Using that $\widehat{\theta} - \theta_0 = O_\proba(1/n)$, and the fact that $\|\partial_\theta^2 \psi(\overline{\theta})\| \leq C$ we deduce that the second term is negligible. Finally, note that $\sum_{i=1}^N  \Delta_i X^{'} \partial_\theta \psi_i(\theta_0)$ can be decomposed as the sum of $\sum_{i=1}^N  \Delta_i \breve{B} \partial_\theta \psi_i(\theta_0)$, where $\breve{B}_t = \int_0^t b_s' ds$, and which is easily proved to be negligible given the local boundedness of $b$ and $\delta$, and $\sum_{i=1}^N  \Delta_i M^c \partial_\theta \psi_i(\theta_0)$, which is a sum of martingale increments with respect to the filtration $\calh_t = \calf_t \vee \sigma \{Q_{t_i}, 1 \leq i \leq N \}$. Thus, by (2.2.35) in \cite{jacod2011discretization}, proving that this term tends to 0 boils down to showing that
\beas 
n^{-1}\sum_{i=1}^N \esp \big[(\Delta_i M^c)^2 \|\partial_\theta \psi_i(\theta_0)\|^2 \big] \to 0,
\eeas 
which is immediate since $\|\partial_\theta \psi_i(\theta_0)^2\| \leq C$, $N/n \to^\proba F$ and $\esp (\Delta_i M^c)^2 \leq C/n$ by Assumption \textbf{(H')}.
\smallskip

We now turn to $II$. As by (\ref{devContinuousTheta}) along with Assumption \textbf{ (H') }, we have for any $k>0$ the inequality  $\proba\l[|\Delta_i \widehat{X}'| > w_i/2\r] \leq Cn^{k(\bar{\omega} - 1/2)}$, we can assume without loss of generality, by taking $k$ sufficiently large, that we can add the indicator $\ind{\mid \Delta_i \widehat{X}'  \mid \leq w_i/2}$ in $II$, i.e. that 
\beas  
II &=& 2n^{1/2} \sum_{i=1}^N \Delta_i \widehat{X}'\Delta_i J  \mathbf{1}_{ \{\mid \Delta_i \widehat{X}  \mid \leq w_i  \}} \ind{\mid \Delta_i \widehat{X}'  \mid \leq w_i/2}, \\
&\leq& 2n^{1/2} \sum_{i=1}^N \Delta_i \widehat{X}'\Delta_i J  \mathbf{1}_{ \{\mid \Delta_i J  \mid \leq 3w_i/2  \}} \ind{\mid \Delta_i \widehat{X}'  \mid \leq w_i/2},
\eeas 
so that 
\beas 
|II| &\leq& 2n^{1/2} \sum_{i=1}^N  |\Delta_i \widehat{X}'| |\Delta_i J|^{1-r}|\Delta_i J|^{r} \mathbf{1}_{ \{\mid \Delta_i J  \mid \leq 3w_i/2  \}} \ind{\mid \Delta_i \widehat{X}'  \mid \leq w_i/2},\\
&\leq& Cn^{1/2 -\bar{\omega}(2-r)} \underbrace{\sum_{i=1}^N |\Delta_i J|^r}_{O_\proba(1)}, 
\eeas  
where we recall that $r >0$ is the jump index of $J$. Given that $\bar{\omega} \in (1/(2(2-r)), 1/2)$, we immediately deduce that $II =o_\proba(1)$. 
Finally, we can show that $III = o_\proba (1)$ with the same line of reasoning as for $II$.

\subsection{Proof of Corollary \ref{volatilitystudent}}
We show (\ref{volatilitystudent2}), as (\ref{volatilitystudent1}) is a particular case where $\phi = 0$. This amounts to proving that $\widehat{AB}$ and $\widehat{AVAR}$ are consistent.

\smallskip
We show first that $\widehat{AB}$ is consistent. As in the previous proofs (in this case this is actually quite easier as we only show the consistency), we can remove the truncation and the parametric noise part and replace $\Delta_i \widehat{X}$ by $\Delta_i  X'$. We obtain that
\begin{eqnarray*}
\widehat{AB} & = & \sum_{i=1}^{B} \frac{2}{3} \overline{v \sigma}_{i} (X_{t_{ih}}' -X_{t_{(i-1)h}}') + o_\proba (1),
\end{eqnarray*}
where  
$$\overline{v \sigma}_{i} = \frac{N^{1/2} \sum_{j=h(i-1)+1}^{hi} (\Delta_j  X^{'})^3}{\sum_{j=h(i-1)+1}^{hi} (\Delta_j X^{'})^2}.$$ 
A Taylor expansion on the function $f(x,y) = x/y$ along with a local version of the convergence (\ref{volatilityass1}), the fact that $\sum_{i=1}^N \l(\Delta_i X^{'}\r)^2 \rightarrow^\proba \Xi$, that $\sigma_t$ and $v_t$ are bounded and bounded away from 0 and that $N/n \rightarrow^\proba F$ yields
\begin{eqnarray*}
\widehat{AB} & = & \sum_{i=1}^{B} \frac{2}{3} v_{t_{i-1}} \sigma_{t_{i-1}} ( X_{t_{ih}}' -  X_{t_{(i-1)h}}') + o_\proba (1).
\end{eqnarray*}
Applying Theorem I.4.31 (iii) on p. 47 in \cite{JacodLimit2003} together with the fact that $\sigma_t$ and $v_t$ are bounded and bounded away from 0, we conclude that $\widehat{AB} \to^\proba AB$. 

\smallskip
We show now that $\widehat{AVAR}$ is consistent. In this case we can again by similar arguments remove the truncation and substitute $\Delta_i \widehat{X}$ by $\Delta_i X'$, i.e. it holds that 
\begin{eqnarray*}
\widehat{AVAR} & = & \frac{2N}{3}  \sum_{i=1}^N (\Delta_i X')^4  - \frac{4}{9} \sum_{i=1}^{B} (\overline{v \sigma}_{i})^2 (X_{t_{ih}}' - X_{t_{(i-1)h}}')^2 + o_\proba (1).
\end{eqnarray*}
By (\ref{volatilityass2}) together with the fact that $N/n \rightarrow^\proba F$, we deduce that 
$$\frac{2N}{3}  \sum_{i=1}^N (\Delta_i X')^4 \rightarrow^\proba \frac{2}{3} \int_0^T u_s \sigma_s^4 ds.$$
Furthermore, using similar techniques as for $\widetilde{AB}$, we obtain that 
$$\frac{4}{9} \sum_{i=1}^{B} (\overline{v \sigma}_{i})^2 (X_{t_{ih}}' - X_{t_{(i-1)h}}')^2 \rightarrow^\proba \frac{4}{9} \int_0^T v_s^2 \sigma_s^4 ds.$$
We have thus shown that $\widehat{AVAR} \to^\proba AVAR$.


\subsection{Proof of Theorem \ref{bipower}}
It is immediate to see that (\ref{bipowereq}) holds as a consequence of (\ref{bipowereq0}) along with Theorem 3.3 in \cite{vetter2010limit}. Accordingly, we show that (\ref{bipowereq0}) holds in what follows, i.e. that
$$n^{1/2} \widehat{\Xi} = n^{1/2} \widetilde{\Xi} + o_\proba(1).$$
First, we show that we can assume without loss of generality that the price process $X$ is continuous, i.e. $J=0$. To do so, we introduce $\widehat{\Xi}^{'}$ as the estimator applied to $X^{'}$ in lieu of $X$. We show that 
\bea 
n^{1/2}\l( \widehat{\Xi} - \widehat{\Xi}'  \r) \to^\proba 0.
\label{removeJumpsBivariate}
\eea 
From (\ref{decompoX0}), we can easily obtain the key decomposition
\bea 
\Delta_i \widehat{X} = \Delta_iX(\widehat{\theta}) = \underbrace{\Delta_i \breve{B} +  \psi_i(\widehat{\theta})}_{\Delta_iB^{'}} + \Delta_i M^c + \Delta_i J,
\label{keyDecompositionBivariate}
\eea 
and by assumption \textbf{(H)}, also recall that we have $ |\psi_i(\widehat{\theta})| \leq \l|\sup_{\theta \in \Theta}\partial_\theta \psi_i(\theta) \r||\widehat{\theta}-\theta_0| \leq C/n$. Thus, remark that all usual conditional moment estimates for $\Delta_i \breve{B}$ are also true for $\Delta_i B^{'}$. More precisely, replacing $\Delta_i \breve{B}$ by $\Delta_i B^{'}$ and $\calf_i$ by $\calg_i = \calf_i \vee \sigma\{Q_{t_i}, 0  \leq i \leq n\}$ in the proof of Lemma 13.2.6 (p. 384) in \cite{jacod2011discretization}, all the conditional estimates are preserved and thus the lemma holds true in the presence of the error term $\psi_i(\widehat{\theta})$. Indeed, the three key ingredients for the original proof of Lemma 13.2.6 are the following (with our own notations): defining 

$$U_i = \frac{|\Delta_i X'|}{\Delta_n^{1/2}}, V_i = \l(\frac{\int_{t_{i-1}}^{t_i} \int_\reels \gamma(z)^{1/r}\mu(ds,dz)}{\Delta_n^{\bar{\omega}}}\r) \wedge 1 \textnormal{ and } W_i = \l(\frac{\int_{t_{i-1}}^{t_i} \int_\reels \gamma(z)^{1/r}\mu(ds,dz)}{\Delta_n^{1/2}}\r) \wedge 1,$$
we have (see (13.2.22)-(13.2.23) in \cite{jacod2011discretization}, pp.384-385), for any $m >0$,
\bea \label{devUi} 
\esp[(U_i)^m | \calg_{i-1}] \leq C_m,
\eea 

\bea 
\esp[(V_i)^m | \calg_{i-1}] \leq \Delta_n^{(1-r\bar{\omega})(1 \wedge (m/r))}\phi_n,
\eea 
\bea \label{devWi}
\esp[(W_i)^m | \calg_{i-1}] \leq \Delta_n^{(1-r/2)(1 \wedge (m/r))}\phi_n,
\eea 
where $C_m >0$ is a constant possibly depending on $m$, and $\phi_n$ is a suitable deterministic sequence tending to $0$ as $n \to +\infty$. Note that in the presence of the term $\psi_i(\widehat{\theta})$, that is, if $V_i$ and $W_i$ are unchanged but $U_i$ is changed to $\widehat{U}_i = \frac{|\Delta_{i}\widehat{X}'|}{\Delta_n^{1/2}}$, the conditional deviations (\ref{devUi})-(\ref{devWi}) remain unchanged since
\beas 
\l|\esp[(U_i)^m | \calg_{i-1}] - \esp[(\widehat{U}_i)^m | \calg_{i-1}]\r| = O_\proba(n^{-1}\Delta_n^{-1/2}) \to 0
\eeas 
using that $\psi_i(\widehat{\theta}) \leq C/n$. Therefore, Lemma 13.2.6 from \cite{jacod2011discretization} still holds when $X$ and $X'$ are respectively replaced by $\widehat{X}$ and $\widehat{X}'$. Applied with $F(x_1,x_2) = |x_1||x_2|$, $k=2$, $p' = s'  =2$, $s=1$ and $\theta = 0$, this directly yields that for all $q \geq 1$ and for some deterministic sequence $a_n$ going to 0,
\beas  
\esp\l||\Delta_i \widehat{X}||\Delta_{i-1} \widehat{X}| \ind{|\Delta_i\widehat{X}| \leq w}\ind{|\Delta_{i-1}\widehat{X}|  \leq w} - |\Delta_i \widehat{X}'||\Delta_{i-1} \widehat{X}'| \ind{|\Delta_i\widehat{X}'| \leq w} \ind{|\Delta_{i-1}\widehat{X}'| \leq w} \r|^q \leq  Ca_n \Delta_n^{(2q-r)\bar{\omega} + 1},
\label{removeJumpBivariate}
\eeas 
where we have used that $q/r > 1$ and $\bar{\omega} <1/2$, and where we recall that $\Delta_i \widehat{X}' =  \Delta_i X'(\widehat{\theta})$. Given the definitions of $\widehat{\Xi}$ and $\widehat{\Xi}'$, applying the above domination with $q=1$, we directly deduce the estimate
\beas 
n^{1/2} \esp |\widehat{\Xi} - \widehat{\Xi}^{'}| &\leq& a_n n^{1/2 - (2-r)\bar{\omega}} \to 0,
\eeas 
since $\bar{\omega} \in (1/(2(2-r)),1/2)$. From now on, by (\ref{removeJumpsBivariate}), we are left to show $n^{1/2}(\widehat{\Xi}' - \widetilde{\Xi}) \to^\proba 0$. By definition, we have that   
\begin{eqnarray*}
n^{1/2} \widehat{\Xi}' & = & \frac{\pi n^{1/2}}{2} \sum_{i=2}^n \big| \Delta_i \widehat{X}' \big| \mathbf{1}_{ \{\mid \Delta_i \widehat{X}'  \mid \leq w  \} } \big| \Delta_{i-1} \widehat{X}' \big| \mathbf{1}_{ \{\mid \Delta_{i-1} \widehat{X}'  \mid \leq w  \} },
\\ & = & \frac{\pi n^{1/2}}{2} \sum_{i=2}^n \big| (\Delta_i X' + \psi_i(\widehat{\theta})) (\Delta_{i-1} X' + \psi_{i-1}(\widehat{\theta})) \big| \mathbf{1}_{ \{\mid \Delta_i \widehat{X}'  \mid \leq w  \} } \mathbf{1}_{ \{\mid \Delta_{i-1} \widehat{X}'  \mid \leq w  \} }.
\end{eqnarray*}
If we introduce $\breve{\Xi}  = \frac{\pi}{2} \sum_{i=2}^n \big| \Delta_i X' \big| \mathbf{1}_{ \{\mid \Delta_i \widehat{X}' \mid \leq w  \} } \big| \Delta_{i-1} X' \big| \mathbf{1}_{ \{\mid \Delta_{i-1} \widehat{X}'  \mid \leq w  \} }$, we have 

\begin{eqnarray*}
 n^{1/2} \big| \widehat{\Xi}' - \breve{\Xi} \big| & = & \frac{\pi n^{1/2}}{2} \sum_{i=2}^n  \big| \Delta_i X'\big|  \mathbf{1}_{ \{\mid \Delta_i \widehat{X}'  \mid \leq w  \} } \l(\big| \Delta_{i-1} \widehat{X}' \big| -\big| \Delta_{i-1} X' \big|  \r)\mathbf{1}_{ \{\mid \Delta_{i-1} \widehat{X}'  \mid \leq w  \} } \\
 &+& \frac{\pi n^{1/2}}{2} \sum_{i=2}^n \l(\big| \Delta_i \widehat{X}' \big| - \big| \Delta_i X'\big| \r)\mathbf{1}_{ \{\mid \Delta_i \widehat{X}'  \mid \leq w  \} } \big| \Delta_{i-1} \widehat{X}' \big| \mathbf{1}_{ \{\mid \Delta_{i-1} \widehat{X}'  \mid \leq w  \} } \\
 &=& I+ II.
\end{eqnarray*}

We prove (\ref{bipowereq0}) in two steps in what follows. First, we show that $n^{1/2} \big| \widetilde{\Xi} - \breve{\Xi} \big| = o_\proba (1)$. Second, we prove that $I = o_\proba (1)$ and $II = o_\proba (1)$. We have 
\beas 
n^{1/2} \big| \widetilde{\Xi} - \breve{\Xi} \big| \leq \frac{\pi n^{1/2}}{2}\sum_{i=2}^n |\Delta_iX'||\Delta_{i-1}X'||\mathbf{1}_{ \{\mid \Delta_i \widehat{X}'  \mid \leq w  \} }\mathbf{1}_{ \{\mid \Delta_{i-1} \widehat{X}'  \mid \leq w  \} } -\mathbf{1}_{ \{\mid \Delta_i  X'  \mid \leq w  \} }\mathbf{1}_{ \{\mid \Delta_{i-1}  X'  \mid \leq w  \} }|,
\eeas 
so that by standard inequalities we can deduce $n^{1/2} \big| \widetilde{\Xi} - \breve{\Xi} \big| \to^\proba 0$ if
\bea 
\esp \l|\mathbf{1}_{ \{\mid \Delta_i \widehat{X}' \mid \leq w  \} } - \mathbf{1}_{ \{\mid \Delta_i X' \mid \leq w  \} }\r|  \leq Cn^{-\beta}
\label{devIndicator}
\eea
for some $\beta >0$ large enough (where $C$ possibly depends on $\beta$). Let us thus show now (\ref{devIndicator}). Introducing $\breve{\Delta}$ as the symmetric difference operator, we have
\beas 
  \l|\mathbf{1}_{ \{\mid \Delta_i \widehat{X}' \mid \leq w  \} } - \mathbf{1}_{ \{\mid \Delta_i X' \mid \leq w  \} }\r|  &=&   \mathbf{1}_{ \{\mid \Delta_i \widehat{X}' \mid \leq w  \} \breve{\Delta} \{\mid \Delta_i  X' \mid \leq w  \}} \\
&\leq&   \ind{|\Delta_iX' -w| \leq |\psi_i(\widehat{\theta})|}\\
&\leq&  \ind{|\Delta_iX' -w| \leq C/n}.\\
\eeas 
Now, letting $\gamma \in (\bar{\omega},1/2)$ and $q >0$, since $\{|\Delta_iX' - w| \leq C/n\} \cap \{|\Delta_iX' | \leq n^{-\gamma}\} = \emptyset$ for $n$ large enough, we automatically have
\beas 
 \ind{|\Delta_iX' - w| \leq C/n} \leq    \ind{|\Delta_iX' | > n^{-\gamma}} \leq n^{\gamma q}|\Delta_i X'|^q,
 \eeas 
hence
 \beas 
\esp \l|\mathbf{1}_{ \{\mid \Delta_i \widehat{X}' \mid \leq w  \} } - \mathbf{1}_{ \{\mid \Delta_i X' \mid \leq w  \} }\r| &\leq& n^{\gamma q}\esp  |\Delta_iX' |^q\\
&\leq& Cn^{q(\gamma-1/2)},
\eeas 
and taking $q$ large enough we get (\ref{devIndicator}). Finally, we prove that $I = o_\proba (1)$. The proof for $II$ is similar. First note that since $X'$ is continuous and $\psi_i(\widehat{\theta}) < K/n$, we can get rid of the indicator functions in $I$ following the same line of reasoning as for (\ref{devXprime}). Moreover, following arguments similar to that of (\ref{devIndicator}), $I$ is asymptotically unaffected if $\mathbf{1}_{\{ |\psi_{i-1}(\widehat{\theta})| < |\Delta_{i-1} X'|\}}$ is present in the sum. Without loss of generality, we can therefore assume that 
\beas 
I = \frac{\pi n^{1/2}}{2} \sum_{i=2}^n  \big| \Delta_i X'\big| \mathbf{1}_{\{ |\psi_{i-1}(\widehat{\theta})| < |\Delta_{i-1} X'|\}}  \l(\big| \Delta_{i-1} \widehat{X}' \big| -\big| \Delta_{i-1} X' \big|  \r) + o_\proba(1).
\eeas 
Next, we decompose $I$ as follows, using the identity for $|y| \leq |x|$, $|x+y|-|x| = y \textnormal{sgn}(x)$ with sgn the usual sign function:
\beas 
I &=&  \frac{\pi n^{1/2}}{2} \sum_{i=2}^n  \psi_{i-1}(\widehat{\theta}) \big| \Delta_i X'\big|  \textnormal{sgn}(\Delta_{i-1} X')    \mathbf{1}_{\{|\psi_{i-1}(\widehat{\theta})| \leq |\Delta_i X'| \}} + o_\proba(1).\\
\eeas  
Again, the indicator function can be removed since its complement event is negligible (it can be majorated by e.g $C|\Delta_i X'|^p/n^p$ for any $p$ where $C$ possibly depends on $p$), which yields the approximation
\beas 
I &=&  \frac{\pi n^{1/2}}{2} \sum_{i=2}^n  \psi_{i-1}(\widehat{\theta}) \big| \Delta_i X'\big|  \textnormal{sgn}(\Delta_{i-1} X')  + o_\proba(1)\\
&=& \frac{\pi n^{1/2} (\widehat{\theta} - \theta_0)^T}{2} \sum_{i=2}^n  \partial_\theta \psi_{i-1}(\theta_0) \big| \Delta_i X'\big|  \textnormal{sgn}(\Delta_{i-1} X')  + o_\proba(1)
\eeas 
where the second step is another application of the mean value theorem (as in the proof of Theorem \ref{volatility}). Now note that standard arguments yield
\beas 
\proba[\textnormal{sgn}(\Delta_{i-1} X') \neq \textnormal{sgn}(\Delta_{i-1} W)] = o_\proba(n^{-p}) 
\eeas 
and 
\beas 
\esp||\Delta_i X'| - |\sigma_{t_{i-2}} \Delta_i W ||^p &\leq& \esp|\Delta_i X' -  \sigma_{t_{i-2}} \Delta_i W|^p\\
&\leq& Cn^{-p}
\eeas 
for any $p >0$ (where the constant $C$ may depend on $p$) and where we have used (\ref{volatilityDefinition}), so that using $\widehat{\theta} - \theta_0 = O_\proba(n^{-1})$ gives 
\beas 
I = \frac{\pi n^{1/2} (\widehat{\theta} - \theta_0)^T}{2} \sum_{i=2}^n \sigma_{t_{i-2}} \partial_\theta \psi_{i-1}(\theta_0) \big| \Delta_i W\big|  \textnormal{sgn}(\Delta_{i-1} W)     + o_\proba(1)\\
\eeas 
which are conditionally centered and uncorrelated increments, with $\textnormal{Var}\l.\l[| \Delta_i W|  \textnormal{sgn}(\Delta_{i-1} W) \r| \calf_{i-2} \r] = O(n^{-1})$,  so that $\sum_{i=2}^n \sigma_{t_{i-2}} \partial_\theta \psi_{i-1}(\theta_0) \big| \Delta_i W\big|  \textnormal{sgn}(\Delta_{i-1} W) = O_\proba(1)$. Therefore, using again that  $\widehat{\theta} - \theta_0 = O_\proba(n^{-1})$, we have $I \to^\proba 0$.

\subsection{Proof of Corollary \ref{bipowerstudent}}
By the stable convergence of Theorem \ref{bipower}, the proof amounts to showing that $\widehat{AVAR}$ is consistent, which is actually a corollary to Theorem \ref{theoremG} in the special case $g(x) = x^2$.

\subsection{Proof of Theorem \ref{covariance}}
Following the discussion at the beginning of Appendix A.2 (p. 30) in \cite{potiron2017estimation} and Proposition 1 from \cite{mykland2009inference}, p. 1408, we can assume without loss of generality that the drift $b_t$ is null as the price process $X$ is continuous.

\smallskip
First, note that (\ref{theorem}) is a straightforward consequence of (\ref{proof201710301135}) together with Theorem 1 (p. 25) in \cite{potiron2017estimation}. Consequently, we only need to show (\ref{proof201710301135}). We now provide the proof of (\ref{proof201710301135}), i.e. that $$\alpha^{-1} \widehat{\Xi} = \alpha^{-1} \widetilde{\Xi} + o_\proba (1).$$
First, note that as a result of Remark 5 (p. 25) in \cite{potiron2017estimation}, $n^{1/2}$ and $\alpha^{-1}$ are of the same order, and thus it is sufficient to show that 
$$n^{1/2} \widehat{\Xi} = n^{1/2} \widetilde{\Xi} + o_\proba (1).$$
Second, we have to reexpress the Hayashi-Yoshida estimator (\ref{HY0}). To do so, we follow the beginning of Section 4.3 in \cite{potiron2017estimation} and introduce some (common) definition in the Hayashi-Yoshida literature. For any positive integer $i$, we consider the 
$i$th sampling time of the first asset $t_{i}^{(1)}$. We define two related random times, $t_{i}^{-}$ 
and $t_{i}^{+}$, which correspond respectively to the closest sampling time of the second asset that is strictly smaller than 
$t_{i}^{(1)}$, and the closest sampling time of the second asset that is (not necessarily strictly) bigger than $t_{i}^{(1)}$. Formally, they are defined as
\begin{eqnarray}
\label{tau-0s} t_{0}^{-} & = & 0, \\
\label{tau-0} t_{i}^{-} & = & \max \{ t_{j}^{(2)} : t_{j}^{(2)} < t_{i}^{(1)} \} \text{ for } i \geq 1, \\
\label{tau+0} t_{i}^{+} & = & \min \{ t_{j}^{(2)} : t_{j}^{(2)} \geq t_{i}^{(1)} \}.
\end{eqnarray}
Rearranging the terms in (\ref{HY0}) gives us
\begin{eqnarray}
\label{HY01}  \widetilde{\Xi} = \sum_{t_{i}^{+} < t} \Delta_i X^{(1)} 
(X_{t_{i}^{+}}^{(2)} - X_{t_{i-1}^{-}}^{(2)}) + o_\proba(n^{-1/2}). 
\end{eqnarray}
We deduce that 
\begin{eqnarray*}
n^{1/2} \widehat{\Xi} & = & n^{1/2} \sum_{t_{i}^{+} < t} \Delta_i \widehat{X}^{(1)} 
(\widehat{X}_{t_{i}^{+}}^{(2)} - \widehat{X}_{t_{i-1}^{-}}^{(2)}) + o_\proba(1), \\
& = & n^{1/2} \widetilde{\Xi} + n^{1/2} \sum_{t_{i}^{+} < t} \psi_i^{(1)} (\widehat{\theta}^{(1)}) \big((\phi(Q_{t_i^+}^{(2)},\theta_0^{(2)} ) - \phi(Q_{t_{i-1}^-}^{(2)},\theta_0^{(2)})) - (\phi(Q_{t_i^+}^{(2)}, \widehat{\theta}^{(2)}) - \phi(Q_{t_{i-1}^-}^{(2)}, \widehat{\theta}^{(2)})) \big)  \\ & & + n^{1/2} \sum_{t_{i}^{+} < t} \Delta_i X^{(1)} \big((\phi(Q_{t_i^+}^{(2)},\theta_0^{(2)} ) - \phi(Q_{t_{i-1}^-}^{(2)},\theta_0^{(2)})) - (\phi(Q_{t_i^+}^{(2)}, \widehat{\theta}^{(2)}) - \phi(Q_{t_{i-1}^-}^{(2)}, \widehat{\theta}^{(2)})) \big) \\
 & & + n^{1/2} \sum_{t_{i}^{+} < t}  \psi_i^{(1)} (\widehat{\theta}^{(1)}) (X_{t_{i}^{+}}^{(2)} - X_{t_{i-1}^{-}}^{(2)}) + o_\proba (1), \\
& := & n^{1/2} \widetilde{\Xi} + I + II + III + o_\proba (1).
\end{eqnarray*}
Our aim is to show that $I = o_\proba (1)$, $II = o_\proba (1)$ and $III = o_\proba (1)$. We start with $I$. On the account that $\phi$ is $C^3$ in $\theta$, and because $\max_i \|Q_{t_i} \|$ is bounded, 
$$I \leq C n^{1/2} N | \widehat{\theta} - \theta_0 |^2,$$
and this is $o_\proba (1)$ by (\ref{thetahat}), Remark 5 (p. 25) and Lemma 8 (p. 31) in \cite{potiron2017estimation}.

\smallskip
As for $II$, the proof of Theorem 2 (p. 46) in \cite{li2016efficient} in the volatility case goes through with one change. To prove (69) in the cited paper, since $$\big((\phi(Q_{t_i^+}^{(2)},\theta_0^{(2)} ) - \phi(Q_{t_{i-1}^-}^{(2)},\theta_0^{(2)})) - (\phi(Q_{t_i^+}^{(2)}, \widehat{\theta}^{(2)}) - \phi(Q_{t_{i-1}^-}^{(2)}, \widehat{\theta}^{(2)})) \big)$$
is not $\mathcal{F}_{t_i}$-measurable, we need to use a Taylor expansion around $\theta_0$. More specifically, let us prove (69) and in line with the notation of the cited paper, we define:
\begin{eqnarray*}
F_N(\theta) & = & \sum_{i=1}^{N^{(1)}} \underbrace{\big((\phi(Q_{t_i^+}^{(2)},\theta ) - \phi(Q_{t_{i-1}^-}^{(2)},\theta)) - (\phi(Q_{t_i^+}^{(2)}, \theta_0^{(2)}) - \phi(Q_{t_{i-1}^-}^{(2)}, \theta_0^{(2)})) \big)}_{\chi_i(\theta)} \underbrace{\int_{t_{i-1}^{(1)}}^{t_i^{(1)}} \sigma_t^{(1)} dW_t^{(1)}}_{\Delta M_i^{c,(1)}}.
\end{eqnarray*}
Note now that by the same Taylor expansion as in (\ref{mvtproof}) and the same line of reasoning, we directly get that for $\theta \in \Theta$ such that $|\theta-\theta_0| \leq K/N$, for some $\overline{\theta} \in [\theta_0, \theta]$,
\beas 
N^l|F_N(\theta) - F_N(\theta_0)|^{2l} \leq C_l N^l |\theta- \theta_0|^{2l}\l(\l|\sum_{i=1}^{N^{(1)}} \partial_\theta \chi_i(\theta_0) \Delta M_i^{c,(1)}\r|^{2l} + \l|\frac{1}{2}\sum_{i=1}^{N^{(1)}} \partial_\theta^2 \chi_i(\overline{\theta}) \Delta M_i^{c,(1)}\r|^{2l} |\theta - \theta_0|^{2l}\r).
\eeas 
Now, using that the first term is a sum of $\calh_t$-martingale increments and Burkholder-Davis-Gundy inequality yields 
\beas 
\esp \l|\sum_{i=1}^{N^{(1)}} \partial_\theta \chi_i(\theta_0) \Delta M_i^{c,(1)}\r|^{2l} &\leq& C \esp \l|\sum_{i=1}^{N^{(1)}} |\partial_\theta \chi_i(\theta_0)|^2  \Delta_i t^{(1)} \r|^l \\
&\leq& C.
\eeas 
Similarly, Jensen inequality applied to the measure $(N^{(1)})^{-1} \sum_{i=1}^{N^{(1)}}$, the boundedness of $|\partial_\theta^2\chi_i(\overline{\theta})|$, and direct calculation of moments for $\Delta M_i^{c, (1)}$ yield 
\beas 
\esp \l|\frac{1}{2}\sum_{i=1}^{N^{(1)}} \partial_\theta^2 \chi_i(\overline{\theta}) \Delta M_i^{c,(1)}\r|^{2l} &\leq& C  N^{2l-1} \esp\sum_{i=1}^{N^{(1)}} \l|\Delta M_i^{c,(1)}\r|^{2l}\\
&\leq& CN^l.
\eeas 
Combined with $|\theta- \theta_0| \leq K/N$, this gives 
\beas 
N^l\esp\sup_{\theta \in \Theta | |\theta-\theta_0| \leq K/N}|F_N(\theta) - F_N(\theta_0)|^{2l} \to 0
\eeas
which is (69) from \cite{li2016efficient}.
Then, one can proceed as in the proof of Theorem 2 (p. 46) in \cite{li2016efficient}.

\smallskip
We turn to $III$, which is slightly more complicated to deal with. We decompose the increment of the second asset in three parts and rewrite $III$ as
\begin{eqnarray*}
III & = & n^{1/2} \Big( \sum_{t_{i}^{+} < t}  \psi_i^{(1)} (\widehat{\theta}^{(1)}) (X_{t_{i}^{+}}^{(2)} - X_{t_{i}^{(1)}}^{(2)}) + \sum_{t_{i}^{+} < t}  \psi_i^{(1)} (\widehat{\theta}^{(1)}) (X_{t_{i}^{(1)}}^{(2)} - X_{t_{i-1}^{(1)}}^{(2)}) + \sum_{t_{i}^{+} < t} \psi_i^{(1)} (\widehat{\theta}^{(1)}) (X_{t_{i-1}^{(1)}}^{(2)} - X_{t_{i-1}^{-}}^{(2)}) \Big)\\
& := & n^{1/2} (III_A + III_B + III_C).
\end{eqnarray*}
The problem with $III_A$ is that it is not adapted to a simple filtration. To circumvent this difficulty, we need to rearrange the terms of the sum again. We follow \cite{potiron2017estimation} (Section 4.3) and we define the new sampling 
times $t_{i}^{1C}$ as
$t_{0}^{1C} := t_{0}^{(1)}$, 
and recursively for $i$ any nonnegative integer
\begin{eqnarray}
\label{algo1C} t_{i+1}^{1C} := \min \big\{ t_{u}^{(1)} : \text{ there exists } j \in \naturels \text{ such that } 
t_{i}^{1C} \leq t_{j}^{(2)} < t_{u}^{(1)} \big\} .
\end{eqnarray}
In analogy with (\ref{tau-0s}), (\ref{tau-0}) and (\ref{tau+0}), we introduce the following times
\begin{eqnarray}
\label{tau1C-0s} t_{0}^{1C,-} & := & 0, \\
\label{tau1C-0} t_{i-1}^{1C,-} & := & \max \{ t_{j}^{(2)} : t_{j}^{(2)} < t_{i-1}^{1C} \} \text{ for } i \geq 2\\
\label{tau1C+0} t_{i-1}^{1C,+} & := & \min \{ t_{j}^{(2)} : t_{j}^{(2)} \geq t_{i-1}^{1C} \} \text{ for } i \geq 1.
\end{eqnarray}
In light of this definition, we can rewrite $III_A$ as
\begin{eqnarray}
\nonumber III_A = \sum_{t_{i}^{1C,+} < t}  \underbrace{\big((\phi(Q_{t_i^{1C}}^{(1)}, \widehat{\theta}^{(1)}) - \phi(Q_{t_{i-1}^{1C}}^{(1)}, \widehat{\theta}^{(1)})) - (\phi(Q_{t_i^{1C}}^{(1)}, \theta_0^{(1)}) - \phi(Q_{t_{i-1}^{1C}}^{(1)}, \theta_0^{(1)})) \big) (X_{t_{i}^{1C,+}}^{(2)} - X_{t_{i}^{1C}}^{(2)})}_{M_{i}(\widehat{\theta}^{(1)})},
\end{eqnarray}
where $M_{i}(\theta)$ is $\mathcal{F}_{t_{i+1}^{1C}}$-measurable. By the mean value theorem, we also have for some $\overline{\theta} \in [\theta_0^{(1)},\widehat{\theta}^{(1)}]$ that  
\begin{eqnarray*}  
n^{1/2}\sum_{i=1}^{N^{(1)}} M_{i}(\widehat{\theta}^{(1)}) = n^{1/2}(\widehat{\theta}^{(1)} -\theta_0^{(1)})^T\sum_{i=1}^{N^{(1)}}  \partial_\theta M_i(\theta_0^{(1)}) + \frac{ n^{1/2}(\widehat{\theta}^{(1)} -\theta_0^{(1)})^T}{2}\sum_{i=1}^{N^{(1)}}  \partial_\theta^2 M_i(\overline{\theta})(\widehat{\theta}^{(1)} -\theta_0^{(1)}).
\end{eqnarray*}
Following the same line of reasoning as for the proof of (\ref{mvtproof}) in the volatility case, we can show that the two terms go to $0$ in probability, so that we have shown that $n^{1/2}III_A = o_\proba (1)$. The other two terms $III_B$ and $III_C$ do not require rearranging the terms. Specifically, $n^{1/2} III_B$ can be shown $o_\proba(1)$ following exactly the proof of Theorem 2 (p. 46) in \cite{li2016efficient}. Regarding the third term $n^{1/2} III_C$, we can show that it is $o_\proba(1)$ using a Taylor expansion similarly as for $III_A$.

\subsection{Proof of Corollary \ref{covariancestudent}}
Although the quantities introduced are quite involved to formally define $\widetilde{AB}$ and $\widetilde{AVAR}$, the proof works the same way as for the proof of (\ref{volatilitystudent2}) in Corollary \ref{volatilitystudent}, along with techniques and estimates from \cite{potiron2017estimation}.

\subsection{Proof of Theorem \ref{theoremG}}

All along this proof, we use the notations $k_n$, $\Delta_n$, $w_n$ in lieu of respectively $k$, $\Delta$ and $w$ in order to emphasize their dependence on $n$. We have to show that $n^{1/2}\l(\widehat{\Xi} - \widetilde{\Xi}'\r) = o_\proba(1)$ where 
\bea  
\widehat{\Xi} = \Delta_n \sum_{i=1}^{[T/\Delta_n]-k_n+1}\l\{g(\widehat{c}_{i}) - \inv{2k_n}\sum_{j,k,l,m=1}^d \partial_{jk,lm}^2g(\widehat{c}_{i})\l(\widehat{c}_{i}^{jl}\widehat{c}_{i}^{km}+\widehat{c}_{i}^{jm}\widehat{c}_{i}^{kl}\r)\r\},
\label{eqGEstimatorRobust}
\eea  
with 
\bea
\widehat{c}_{i}^{lm}= \inv{k_n\Delta_n}\sum_{j=0}^{k_n-1}\Delta_{i+j} \widehat{X}^l\Delta_{i+j} \widehat{X}^m \mathbf{1}_{\{\|\Delta_{i+j} \widehat{X}\| \leq w_n\}}. 
\label{defcHat}
\eea
We start by showing that we can assume without loss of generality that $X$ is continuous, i.e replace $X$ by $X'$ in all the expressions. To do so, consider $\widehat{\Xi}'$ and $\widehat{c}_i'$ the estimators applied to the continuous part $X'$ in lieu of $X$. Without loss of generality, we assume in what follows that $X$, $\widehat{\theta}$ and $\theta_0$ are $1$-dimensional quantities. The multi-dimensional case can be derived by a straightforward adaptation.  

\begin{lemma*}\label{removeJumpG}
We have 
$$ n^{1/2}\l(\widehat{\Xi} - \widehat{\Xi}^{'}\r) \to^\proba 0.$$
\end{lemma*}

\begin{proof}
Recall that we have the key decomposition
\bea 
\Delta_i \widehat{X} = \Delta_iX(\widehat{\theta}) = \underbrace{\Delta_i\breve{B} +  \psi_i(\widehat{\theta})}_{\Delta_iB^{'}} + \Delta_i M^c + \Delta_i J,
\label{keyDecomposition}
\eea 
where we recall that $\breve{B}_t = \int_0^t b_s' ds$. Now, we apply exactly the same line of reasoning as for the proof of Theorem \ref{bipower}. We replace again $\Delta_i \breve{B}$ by $\Delta_i B^{'}$ and $\calf_i$ by $\calg_i = \calf_i \vee \sigma\{Q_{t_i}, 0  \leq i \leq n\}$ in the proof of Lemma 13.2.6 (p. 384) in \cite{jacod2011discretization}, all the conditional estimates are preserved and thus the lemma remains valid in the presence of the term $\psi_i(\widehat{\theta})$. Applied with $F(x) = x^2$, $k=1$, $p' = s'  =2$, $s=1$ and $\theta = 0$, this directly yields that for all $q \geq 1$ and for some deterministic sequence $a_n$ shrinking to 0, we have that 
\bea 
\esp\l||\Delta_i \widehat{X}|^2 \ind{|\Delta_i\widehat{X}| \leq w_n} - |\Delta_i \widehat{X}'|^2 \ind{|\Delta_i\widehat{X}'| \leq w_n} \r|^q \leq  Ca_n \Delta_n^{(2q-r)\bar{\omega} + 1}.
\label{removeJump1}
\eea
As a by-product, we also deduce
\bea 
\esp\l| \widehat{c}_i - \widehat{c}_i'\r|^q \leq C a_n \Delta_n^{(2q-r)\bar{\omega} + 1 - q}.
\label{removeJump2}
\eea 
 Moreover, replacing again $\calf_i$ by $\calg_i$ and $\Delta_i \breve{B}$ by $\Delta_i B'$ in the calculation we can also see that the second inequality of (4.10) in \cite{jacod2013quarticity} remains true in the presence of $\psi_i(\widehat{\theta})$, that is, introducing $\alpha_i = |\Delta_i \widehat{X}'|^2 - \sigma_{t_i}^2 \Delta_n$, we have 
\bea 
\l|\esp[\alpha_i | \calg_i]\r| \leq C \Delta_n^{3/2}.
\label{removeJump3}
\eea 
Now, remark that by the proof of Lemma 4.4 (p. 1479, case $v=1$) in \cite{jacod2013quarticity},  $ n^{1/2}\big(\widehat{\Xi} - \widehat{\Xi}^{'}\big) \to^\proba 0$ is an immediate consequence of our estimates (\ref{removeJump2}) and (\ref{removeJump3}), along with the polynomial condition (\ref{momentG}) on $g$.
\end{proof}

From now on, by virtue of Lemma \ref{removeJumpG}, we only have to prove $n^{1/2}(\widehat{\Xi}' - \widetilde{\Xi}') \to^\proba 0$. We now want to show that in the definition of $\widehat{\Xi}'$, we can substitute $\widehat{c}_{i}'$ by $\overline{c}_{i}'$, where 
\bea
\overline{c}_{i}'^{lm}= \inv{k_n\Delta_n}\sum_{j=0}^{k_n-1}\Delta_{i+j} \widehat{X}'^l\Delta_{i+j} \widehat{X}'^m \mathbf{1}_{\{|\Delta_{i+j} X'| \leq w_n\}},
\label{defcBar}
\eea
that is when the indicator function is applied to $X'$ itself instead of $\widehat{X}'$. We first state a technical lemma.

\begin{lemma*}\label{lemmaBoundc}
We have, for any $i \in \{1,\cdots,n\}$, any $j \in \{1,\cdots,3\}$, and any $q \geq 1$, 
\beas 
  \esp|\partial^j g(\widehat{c}_i')|^q  \leq C \textnormal{ and } \esp|\partial^j g(\overline{c}_i')|^q  \leq C.
\eeas

\end{lemma*}

\begin{proof}
In view of (\ref{momentG}), it is sufficient to prove that for any $q \geq 1$,
\beas 
\esp|  \widehat{c}_i' |^q \leq C\textnormal{ and }\esp|\overline{c}_i' |^q \leq C.
\eeas 
Moreover, since $|  \widehat{c}_i' |^q \leq C( |\widehat{c}_i'  -  \overline{c}_i'|^q + |\overline{c}_i'  -  \widetilde{c}_i|^q +|\widetilde{c}_i |^q  )$, and as $\esp |\widetilde{c}_i|^q \leq C$ as an easy consequence of (4.11) in \cite{jacod2013quarticity} (p. 1476) and the boundedness of $c$ in Assumption \textbf{(H)}, it suffices to show the $\mathbb{L}_q$ boundedness of 
 \begin{eqnarray}
 \label{proof201712001924}
 \widehat{c}_i'  -  \overline{c}_i' = \inv{k_n\Delta_n}\sum_{j=0}^{k_n-1}|\Delta_{i+j} \widehat{X}'|^2\l(\mathbf{1}_{\{|\Delta_{i+j} \widehat{X}'| \leq w_n\}} - \mathbf{1}_{\{|\Delta_{i+j}  X'| \leq w_n\}}\r)
 \end{eqnarray}
and 
\begin{eqnarray}
  \overline{c}_i'-\widetilde{c}_i   &\leq&  \frac{2}{k_n\Delta_n}\sum_{j=0}^{k_n-1}\Delta_{i+j}X'\psi_{i+j}(\widehat{\theta})\mathbf{1}_{\{|\Delta_{i+j}  X'| \leq w_n\}} + \inv{k_n\Delta_n}\sum_{j=0}^{k_n-1} \psi_{i+j}(\widehat{\theta})^2, \label{proof20171201925}\\
&:=& I + II. \nonumber
\end{eqnarray}
We first show the $\mathbb{L}_q$ boundedness of (\ref{proof201712001924}). First recall that in (\ref{devIndicator}) we proved that
\beas  
\esp \l|\mathbf{1}_{ \{\mid \Delta_i \widehat{X}' \mid \leq w_n  \} } - \mathbf{1}_{ \{\mid \Delta_i X' \mid \leq w_n  \} }\r|  \leq n^{-\beta}
\eeas
for any $\beta >0$. 
Thus, by Cauchy-Schwarz inequality and Jensen's inequality we easily get that $\esp |\widehat{c}_i'  -  \overline{c}_i'|^q \leq C$ considering $\beta$ large enough. 

We prove now the $\mathbb{L}_q$ boundedness of (\ref{proof20171201925}). By Jensen's inequality applied to $$|k_n^{-1}\sum_{j=0}^{k_n-1}\Delta_{i+j}X'\psi_{i+j}(\widehat{\theta})|^q,$$ we have 
\beas 
\esp |I|^q &\leq& \frac{Cn^{q}}{k_n} \sum_{j=0}^{k_n-1}\esp|\Delta_{i+j}X'|^q\underbrace{|\psi_{i+j}(\widehat{\theta})|^q}_{C/n^q}\\
&\leq& C n^{-q/2}.
\eeas 
For $II$ we have 
\beas 
\esp |II|^q &\leq& \frac{Cn^{q}}{k_n}\esp\sum_{j=0}^{k_n-1} |\psi_{i+j}(\widehat{\theta})|^{2q}\\
&\leq& C n^{-q},
\eeas 
and thus this yields the $\mathbb{L}_q$ boundedness of $\overline{c}_i'-\widetilde{c}_i$, which concludes the proof. 
\end{proof}

\begin{lemma*}\label{lemmaReplaceIndicator}
Let $\overline{\Xi}'$ be defined as $\widehat{\Xi}'$ where $\widehat{c}_{i}'$ is replaced by $\overline{c}_{i}'$. Then 
$$ n^{1/2}\l(\widehat{\Xi}' - \overline{\Xi}'\r) \to^\proba 0.$$
\end{lemma*}

\begin{proof}
We have 
\begin{eqnarray} 
\label{proof201712011018}n^{1/2}\l(\widehat{\Xi}' - \overline{\Xi}'\r) &=& n^{1/2}\Delta_n \sum_{i=1}^{[T/\Delta_n]-k_n+1}\l\{g(\widehat{c}_{i}') - g(\overline{c}_{i}') \r\} \\
\nonumber &+& \frac{n^{1/2}\Delta_n}{2k_n}\sum_{i=1}^{[T/\Delta_n]-k_n+1} \l\{ h(\overline{c}_{i}')- h(\widehat{c}_{i}')\r\},
\end{eqnarray}
with $h(x) =   2\partial^2g(x)x^2$, so that proving our claim boils down to showing that both terms in the right-hand side of (\ref{proof201712011018}) are negligible. For the first one, we have  
\beas 
\sum_{i=1}^{[T/\Delta_n]-k_n+1}\l|g(\widehat{c}_{i}') - g(\overline{c}_{i}')  \r| &\leq&   \inv{k_n\Delta_n}\sum_{i=1}^{[T/\Delta_n]-k_n+1}\sum_{j=0}^{k_n-1} \l|\partial g(a_{i,j})\r|   |\Delta_{i+j} \widehat{X}'|^2 \l|\mathbf{1}_{\{\l|\Delta_{i+j} \widehat{X}'\r| \leq w_n\}} - \mathbf{1}_{\{\l|\Delta_{i+j} X'\r| \leq w_n\}}\r|   
\eeas 
for some (random) $a_{i,j}$ such that $|a_{i,j}| \leq |\widehat{c}_i'| +  |\overline{c}_i'|$ by the mean value theorem. Now, by Lemma \ref{lemmaBoundc} and the fact that $g$ is of polynomial growth we get $\esp|\partial g(a_{i,j})|^q \leq C$ for any $q \geq 1$, and thus by   Cauchy-Schwarz inequality we will have 
$$n^{1/2}\Delta_n \sum_{i=1}^{[T/\Delta_n]-k_n+1}\l\{g(\widehat{c}_{i}') - g(\overline{c}_{i}') \r\} \to^\proba 0$$
if we can prove that     
$$ \sum_{i=1}^{[T/\Delta_n]-k_n+1}\sum_{j=0}^{k_n-1} \l(\esp\l[|\Delta_{i+j} \widehat{X}'|^{4} \l|\mathbf{1}_{\{\l|\Delta_{i+j} \widehat{X}\r| \leq w_n\}} - \mathbf{1}_{\{|\Delta_{i+j} X'| \leq w_n\}}\r|\r]\r)^{1/2} = o(k_nn^{-1/2}),$$
i.e. that 
$$ \sum_{i=1}^{[T/\Delta_n]-k_n+1}\l(\esp\l[|\Delta_{i } \widehat{X}'|^{4} \l|\mathbf{1}_{\{\l|\Delta_{i } \widehat{X}'\r| \leq w_n\}} - \mathbf{1}_{\{|\Delta_{i} X'| \leq w_n\}}\r|\r]\r)^{1/2} = o( n^{-1/2}).$$
Recalling $|\Delta_i\widehat{X}'|^{4} \leq C(|\Delta_i X'|^{4} +   |\psi_i(\widehat{\theta}) |^{4} )$, we have that 
$$ \sum_{i=1}^{[T/\Delta_n]-k_n+1}  \l(\esp\l[|\Delta_{i }  X'|^{4} \l|\mathbf{1}_{\{\l|\Delta_{i } \widehat{X}'\r| \leq w_n\}} - \mathbf{1}_{\{|\Delta_{i} X'| \leq w_n\}}\r|\r]\r)^{1/2} = O (n^{   -\beta/4}) = o( n^{-1/2})$$ since $\beta$ can be taken arbitrary big, using again Cauchy-Schwarz inequality along with the fact that $\esp |\Delta_{i } X'|^q\leq C n^{-q/2}$, and (\ref{devIndicator}). 
Finally, it is immediate to prove
$$ \sum_{i=1}^{[T/\Delta_n]-k_n+1}\l(\esp\l[|\psi_i(\widehat{\theta})|^{4} \l|\mathbf{1}_{\{|\Delta_{i } \widehat{X}'| \leq w_n\}} - \mathbf{1}_{\{|\Delta_{i } X'| \leq w_n\}}\r|\r]\r)^{1/2} = o\l(n^{-1/2}\r),$$
given that $|\psi_i(\widehat{\theta})|^{4} \leq K/n^4$. The second term on the right-hand side of (\ref{proof201712011018}) is proved in the same way.
\end{proof}
 In the $1$-dimensional setting, we now introduce the following notation for $\theta \in \Theta$:
$$ c_i'(\theta) = \inv{k_n\Delta_n}\sum_{j=0}^{k_n-1}|\Delta_{i+j} X'(\theta)|^2 1_{\{|\Delta_{i+j} X' | \leq w_n\}},$$
where we recall that for any $i \in \{1,\cdots,n\}$, $\Delta_iX'(\theta) = \Delta_iX' + \psi_i(\theta)$. Note that $\overline{c}_{i}'  = c_i'(\widehat{\theta})$, and $\widetilde{c}_{i}  = c_i'( \theta_0 )$. We define 
$$E_n := n^{1/2}\Delta_n\sum_{i=1}^{[T/\Delta_n]-k_n+1}\l\{g(\overline{c}_{i}') - g(\widetilde{c}_{i}) \r\}.$$
By the mean value theorem along with the chain rule we have for some $\overline{\theta} \in [\theta_0,\widehat{\theta}]$,  
\beas 
E_n &=& \frac{2n^{1/2}}{k_n }(\widehat{\theta} - \theta_0)\sum_{i=1}^{[T/\Delta_n]-k_n+1} \partial g(\widetilde{c}_{i})\sum_{j=0}^{k_n-1} \Delta_{i+j} X' \partial_\theta \psi_{i+j}(\theta_0)1_{\{|\Delta_{i+j} X'| \leq w_n\}}\\
&+& \frac{n^{1/2}}{k_n}(\widehat{\theta} - \theta_0)^2\sum_{i=1}^{[T/\Delta_n]-k_n+1}\partial g(c_i'(\overline{\theta}))\sum_{j=0}^{k_n-1}\Delta_{i+j} X'(\overline{\theta}) \partial_\theta^2 \psi_{i+j}(\overline{\theta})1_{\{|\Delta_{i+j} X' | \leq w_n\}} \\
&+& \frac{n^{1/2}}{k_n}(\widehat{\theta} - \theta_0)^2\sum_{i=1}^{[T/\Delta_n]-k_n+1}\partial g(c_i'(\overline{\theta}))\sum_{j=0}^{k_n-1}  \partial_\theta \psi_{i+j}(\overline{\theta})^21_{\{|\Delta_{i+j} X'| \leq w_n\}}\\
&+&\frac{2n^{1/2}}{k_n^2\Delta_n}(\widehat{\theta} - \theta_0)^2 \sum_{i=1}^{[T/\Delta_n]-k_n+1}\partial^2g(c_i'(\overline{\theta}))\l\{\sum_{j=0}^{k_n-1} \Delta_{i+j} X'(\overline{\theta}) \partial_\theta \psi_{i+j}(\overline{\theta})1_{\{|\Delta_{i+j} X' | \leq w_n\}}\r\}^2,\\
&:=&I + II + III + IV.
\eeas 
We now show that each term is $o_\proba(1)$.

\begin{lemma*}\label{lemmaTermI}
We have 
$$I = \frac{2n^{1/2}}{k_n }(\widehat{\theta} - \theta_0)\sum_{i=1}^{[T/\Delta_n]-k_n+1} \partial g(\widetilde{c}_{i})\sum_{j=0}^{k_n-1} \Delta_{i+j} X' \partial_\theta \psi_{i+j}(\theta_0)1_{\{|\Delta_{i+j} X' | \leq w_n\}} \to^\proba 0.$$
\end{lemma*}

\begin{proof}
Since Assumption \textbf{(H)} yields $\frac{2n^{1/2}}{k_n }(\widehat{\theta} - \theta_0) = O_\proba(k_n^{-1}n^{-1/2})$, it suffices to prove that 
\bea 
\sum_{i=1}^{[T/\Delta_n]-k_n+1} \partial g(\widetilde{c}_{i})\sum_{j=0}^{k_n-1} \Delta_{i+j} X' \partial_\theta \psi_{i+j}(\theta_0)1_{\{|\Delta_{i+j} X | \leq w_n\}} = o_\proba(k_n n^{1/2}). 
\label{convTermDeriv}
\eea 
Recalling the decomposition $\Delta_{i+j}X' = \Delta_{i+j}\breve{B} +\Delta_{i+j}M^c$, 
we first show that the above term is negligible when $\Delta_{i+j}X'$ is replaced by $\Delta_{i+j}M^c$. In that case, by virtue of the domination $1_{\{|\Delta_{i+j} M^c | \geq w_n\}} \leq w_n^{-1}|\Delta_{i+j} M^c|$, Burkh\"{o}lder-Davis-Gundy inequality, H\"{o}lder's inequality, along with the fact that $|\partial g(\widetilde{c}_{i})|$ is $\mathbb{L}_q$ bounded by Lemma \ref{lemmaBoundc}, the indicator function can be removed without loss of generality. Thus, introducing 
$$A_n = \sum_{i=1}^{[T/\Delta_n]-k_n+1} \partial g(\widetilde{c}_{i})\sum_{j=0}^{k_n-1} \Delta_{i+j} M^c \partial_\theta \psi_{i+j}(\theta_0), $$
and
$$B_n = \sum_{i=1}^{[T/\Delta_n]-k_n+1} \partial g(c_{t_i})\sum_{j=0}^{k_n-1} \Delta_{i+j} M^c \partial_\theta \psi_{i+j}(\theta_0), $$
we show that $A_n - B_n = o_\proba(k_n n^{1/2})$ and $B_n = o_\proba(k_n n^{1/2})$ separately. We have for some $\xi_i \in [\widetilde{c}_i,c_{t_i}]$,
\beas |A_n - B_n| &\leq&   \sum_{i=1}^{[T/\Delta_n]-k_n+1} \l|\partial^2 g(\xi_i)\r| \l| \widetilde{c}_{i}- c_{t_i}\r|\sum_{j=0}^{k_n-1} |\Delta_{i+j} M^c| |\partial_\theta \psi_{i+j}(\theta_0)|.
\eeas 
Moreover, by (4.11) in \cite{jacod2013quarticity} (p. 1476), we have the estimate
\bea 
\esp\l[\l|\widetilde{c}_i - c_{t_i} \r|^2\r] \leq C\l(k_n^{-1} + k_n\Delta_n\r).
\label{eqDevVol}
\eea 
Thus, by application of H\"{o}lder's inequality, the fact that $\partial^2 g(\xi_i)$ is $\mathbb{L}_q$ bounded by Lemma \ref{lemmaBoundc}, and that for any $q\geq1$:
\beas 
\esp \l[|\Delta_{i+j} M^c|^q |\partial_\theta \psi_{i+j}(\theta_0)|^q\r] &\leq&   C \esp[|\Delta_{i+j} M^c|^q]\\
&\leq& C n^{-q/2},
\eeas 
we deduce that 
\beas 
\esp|A_n - B_n| \leq Ck_nn^{1/2}\l(k_n^{-1} + k_n\Delta_n\r)^{1/2} = o_\proba(k_nn^{1/2}).
\eeas 
As for $B_n$, we note that it can be expressed as a sum of martingale increments with respect to the filtration $\calh_t = \calf_t \vee \sigma\{Q_{t_i}, i =0,\cdots,n\}$, and we have $B_n = \sum_{i=1}^{[T/\Delta_n]} \chi_i$ with 
$$\chi_i =  \sum_{l= (i-k_n+1) \wedge 1}^i \partial g(\sigma_{t_l}^2) \partial_\theta\psi_{i}(\theta_0) \Delta_{i} M^c.$$ 
Thus, by property (2.2.35) p. 56 in \cite{jacod2011discretization}, proving that $B_n = o_\proba(k_nn^{1/2})$ boils down to showing that 
\bea 
\widetilde{B}_n := n^{-1 }k_n^{-2}\sum_{i=1}^{[T/\Delta_n] } \esp\chi_i^2 \to  0.
\label{eqMartBtilde}
\eea 
Now, using the boundedness of $c$, we have 
\beas
\esp \chi_i^2 &\leq& Ck_n^2 \esp \partial_\theta\psi_{i}(\theta_0)^2 \l(\Delta_{i} M^c\r)^2 \\
&\leq& C k_n^2 n^{-1}. 
\eeas 
Therefore $\widetilde{B}_n = O_\proba(n^{-1})$ which proves (\ref{eqMartBtilde}) and thus (\ref{convTermDeriv}) when replacing $\Delta_{i+j}X'$ by $\Delta_{i+j}M^c$. 
Finally, the case where we consider the drift term $\Delta_{i+j} \breve{B}$ in lieu of $\Delta_{i+j} X'$ follows immediately from the fact that $\esp |\Delta_{i+j} \breve{B}|^k \leq C n^{-k} $ for any $k \geq 1$.
\end{proof}

\begin{lemma*}\label{lemmaTermII}
We have that $II = o_\proba(1)$, $III = o_\proba(1)$, $IV = o_\proba(1)$.
\end{lemma*}

\begin{proof}
Proving the first claim is equivalent to showing that 
\beas
\widetilde{II} := \sum_{i=1}^{[T/\Delta_n]-k_n+1}\partial g(c_i'(\overline{\theta}))\sum_{j=0}^{k_n-1}\Delta_{i+j} X'(\overline{\theta}) \partial_\theta^2 \psi_{i+j}(\overline{\theta})1_{\{|\Delta_{i+j} X' | \leq w_n\}}  = o_\proba(k_nn^{3/2}).
\eeas 
Note, again, that by Assumption \textbf{(H)} and the fact that $\overline{\theta}$ belongs to a compact set, we have $|\partial_\theta^2\psi_{i+j}(\overline{\theta})| \leq C$. Thus 
\beas 
\esp\l|\widetilde{II}\r| &\leq& C\sum_{i=1}^{[T/\Delta_n]-k_n+1}\esp\l[|\partial g(c_i'(\overline{\theta}))|\sum_{j=0}^{k_n-1}|\Delta_{i+j} X'(\overline{\theta})| \r]\\
&\leq& C \sum_{i=1}^{[T/\Delta_n]-k_n+1}\sum_{j=0}^{k_n-1}\l(\esp\partial g(c_i'(\overline{\theta}))^2\r)^{1/2}\l(\esp|\Delta_{i+j} X'(\overline{\theta})|^2 \r)^{1/2}\\
&\leq& C k_n n^{1/2} = o_\proba(k_n n^{3/2}),
\eeas 
where we have used Lemma \ref{lemmaBoundc}, and the fact that for any $q \geq 1$,  
\bea 
\esp |\Delta_{i+j}X'(\overline{\theta})|^q \leq C\l(\esp |\Delta_{i+j}X'|^q +\esp\l[ \underbrace{(\overline{\theta}-\theta_0)^q}_{\leq K/n^q} \underbrace{\sup_{\theta \in \Theta}|\partial_\theta \psi_i(\theta)|^q}_{\leq K}\r] \r) \leq C\l(n^{-  q/2} + n^{-q}\r).
\label{devXtheta}
\eea 
For the second claim, we have (bounding the indicator function from above by $1$) the estimate 
\beas 
\widetilde{III} &\leq& \sum_{i=1}^{[T/\Delta_n]-k_n+1}\partial g(c_i'(\overline{\theta}))\sum_{j=0}^{k_n-1}  \partial_\theta \psi_{i+j}(\overline{\theta})^2\\
&\leq& C k_n\underbrace{\sum_{i=1}^{[T/\Delta_n]-k_n+1}|\partial g(c_i'(\overline{\theta}))|}_{O_\proba(n)}\\
&=& O_\proba(k_n n) = o_\proba(k_n n^{3/2}),
\eeas 
so that $III=o_\proba(1)$. Finally we show that $IV = o_\proba(1)$, that is
\bea 
\widetilde{IV} := \sum_{i=1}^{[T/\Delta_n]-k_n+1}\partial^2g(c_i'(\overline{\theta}))\l\{\sum_{j=0}^{k_n-1} \Delta_{i+j} X'(\overline{\theta}) \partial_\theta \psi_{i+j}(\overline{\theta})1_{\{|\Delta_{i+j} X' | \leq w_n\}}\r\}^2 = o_\proba(k_n^2 n^{1/2}).
\label{estIV}
\eea 
By Cauchy-Schwarz inequality and the fact that $|\partial_\theta\psi_{i+j}(\overline{\theta})|^2 \leq C$, we get the domination
\beas 
\esp|\widetilde{IV}| &\leq& Ck_n\esp\l[\sum_{i=1}^{[T/\Delta_n]-k_n+1}|\partial^2g(c_i'(\overline{\theta}))|\sum_{j=0}^{k_n-1} |\Delta_{i+j} X'(\overline{\theta})|^2\r]\\
&\leq&Ck_n \sum_{i=1}^{[T/\Delta_n]-k_n+1}\sum_{j=0}^{k_n-1} \underbrace{\l(\esp\partial^2g(c_i'(\overline{\theta}))^{2} \r)^{1/2}}_{\leq C} \underbrace{\l(\esp|\Delta_{i+j} X'(\overline{\theta})|^{4}\r)^{1/2}}_{O(n^{-1})}\\
&\leq& Ck_n^2 = o(k_n^2n^{1/2}), 
\eeas 
where we have used (\ref{devXtheta}) with $q = 4$, and we are done. 
\end{proof}
Similarly we have by the mean value theorem that
$$\frac{n^{1/2}\Delta_n}{k_n}\sum_{i=1}^{[T/\Delta_n]-k_n+1}\l\{h(\overline{c}_{i}') - h(\widetilde{c}_{i}) \r\}$$ 
is equal to
$$\frac{2n^{1/2}}{k_n^2 }(\widehat{\theta} - \theta_0)\sum_{i=1}^{[T/\Delta_n]-k_n+1} \partial h( c_{i}'(\overline{\theta}))\sum_{j=0}^{k_n-1} \Delta_{i+j} X'(\overline{\theta}) \partial_\theta \psi_{i+j}(\overline{\theta})1_{\{|\Delta_{i+j} X' | \leq w_n\}}.$$

\begin{lemma*}\label{lemmaH}
We have 
\beas 
  \frac{n^{1/2}\Delta_n}{k_n}\sum_{i=1}^{[T/\Delta_n]-k_n+1}\l\{h(\overline{c}_{i}') - h(\widetilde{c}_{i}) \r\} \to^\proba 0.
\eeas 
\end{lemma*}
\begin{proof}
By Assumption \textbf{(H)} we have 
\beas 
  \esp\l|\frac{n^{1/2}\Delta_n}{k_n}\sum_{i=1}^{[T/\Delta_n]-k_n+1}\l\{h(\overline{c}_{i}') - h(\widetilde{c}_{i}) \r\}\r| &\leq&  \frac{C}{n^{
  1/2}k_n^2 } \sum_{i=1}^{[T/\Delta_n]-k_n+1} \sum_{j=0}^{k_n-1} \esp \l[|\partial h( c_{i}'(\overline{\theta}))||\Delta_{i+j} X'(\overline{\theta})|\r].
\eeas
Since $\partial h$ is also of polynomial growth, we deduce as for Lemma \ref{lemmaBoundc} that for any $q \geq 1$, $\esp|\partial h(c_i'(\overline{\theta}))|^q \leq C$, and so an application of Cauchy-Schwarz inequality yields
\beas 
\esp\l|\frac{n^{1/2}\Delta_n}{k_n}\sum_{i=1}^{[T/\Delta_n]-k_n+1}\l\{h(\overline{c}_{i}) - h(\widetilde{c}_{i}) \r\}\r| \leq C/k_n \to0.
\eeas 
\end{proof}
We prove now the theorem.

\begin{proof}[Proof of Theorem \ref{theoremG}.] 
Recall that by Lemma \ref{removeJumpG} we only need to prove that $n^{1/2}( \widehat{\Xi}^{'} -\widetilde{\Xi}') \to^\proba0$. We have 
$$ n^{1/2}\l(\widehat{\Xi}' - \widetilde{\Xi}'\r) = n^{1/2} \l(\widehat{\Xi}' - \overline{\Xi}'\r) + n^{1/2} \l(\overline{\Xi}' - \widetilde{\Xi}' \r). $$
The first term above is negligible by virtue of Lemma \ref{lemmaReplaceIndicator}. Moreover, since
\beas 
n^{1/2}\l(\overline{\Xi}' - \widetilde{\Xi}'\r) &=& n^{1/2}\Delta_n \sum_{i=1}^{[T/\Delta_n]-k_n+1}\l\{g(\overline{c}_{i}') - g(\widetilde{c}_{i}) \r\} \\
&+& \frac{n^{1/2}\Delta_n}{2k_n}\sum_{i=1}^{[T/\Delta_n]-k_n+1} \l\{ h(\widetilde{c}_{i})- h(\overline{c}_{i}')\r\},
\eeas
the assertion $n^{1/2}\l(\overline{\Xi}' - \widetilde{\Xi}'\r) \to^\proba 0$ is an immediate consequence of Lemma \ref{lemmaTermI}, Lemma \ref{lemmaTermII} and Lemma \ref{lemmaH}. Combined with Theorem 3.2  (p. 1469, applied to $X^{'}$) in \cite{jacod2013quarticity}, this yields the central limit theorem.
\end{proof}

\subsection{Proof of Corollary \ref{corolStudentG}} 
By Slutsky's Lemma, all we need to prove is that $\widehat{AVAR} \to^\proba AVAR$. Given the form of $\widehat{AVAR}$, this can be shown using exactly the same line of reasoning as for the general theorem replacing $g$ by $\overline{h}$ in all our estimates and combining the results with Corollary 3.7 in \cite{jacod2013quarticity} in lieu of Theorem 3.2, except that there is no scaling by $n^{1/2}$ in front of the estimates and no bias term. Since the $C^3$ property of $g$ is only used once when handling the bias term in Lemma  \ref{lemmaH}, the fact that $\overline{h}$ is only of class $C^2$ is not problematic. 

\subsection{Proof of Theorem \ref{thmVolofVol} and Corollary \ref{corVolofVol}}

In \cite{vetter2015estimation}, the author introduces 
\beas 
A_i = \frac{2n}{k_n}\sum_{j=1}^{k_n} \int_{(i+j-1)T/n}^{(i+j)T/n}(X_s-X_{(i+j-1)T/n})dX_s
\eeas 
and 
\beas
B_i := \frac{n}{k_n}\int_{iT/n}^{(i+k_n)T/n}\sigma_s^2ds.
\eeas 
Accordingly, we define
\beas 
\widehat{A}_i &:=& \frac{2n}{k_n}\sum_{j=1}^{k_n}\l\{\int_{(i+j-1)T/n}^{(i+j)T/n}(X_s-X_{(i+j-1)T/n})dX_s +  \psi_{i+j}(\widehat{\theta})\Delta_{i+j}X    \r\},\\
 \widehat{B}_i &:=& \frac{n}{k_n}\l\{\int_{iT/n}^{(i+k_n)T/n}\sigma_s^2ds+ \sum_{j=1}^{k_n}\psi_{i+j}(\widehat{\theta})^2\r\} ,
\eeas 
along with the approximated increments for some arbitrary $p \geq 1$ and  $1 \leq l \leq J(p) := [[nt/T - 2k_n]/((p+2)k_n)]$, where $[x]$ is defined as the floor function of $x$,
\beas 
\widetilde{A}_{i+k_n} - \widetilde{A}_{i} := \frac{n}{k_n}\sigma_{a_l(p)T/n}\sum_{j=1}^{k_n} \l(\Delta_{i+k_n+j}W^2 - \Delta_{i+j}W^2\r), 
\eeas 
and
\beas 
\widetilde{B}_{i+k_n} - \widetilde{B}_{i} := \frac{n}{k_n}\int_{iT/n}^{(i+k_n)T/n} \widetilde{\sigma}_{a_l(p)T/n}(W_{(s+k_nT/n)}^{'} - W_{ s }^{'})ds,
\eeas 
where $a_l(p) := (l-1)(p+2)k_n$. Note that $\widehat{c}_{i} = \widehat{A}_i + \widehat{B}_i$, and therefore $\widehat{\Xi}$ can be linked to the above quantities as follows:
\bea \label{volOfVolKeyEquation}
\widehat{\Xi} = \sum_{i=0}^{[T/\Delta_n]-2k_n} \l\{\frac{3}{2k_n}(\widehat{A}_{i+k_n} - \widehat{A}_i + \widehat{B}_{i+k_n} - \widehat{B}_i)^2 - \frac{6}{k_n^2}\widehat{q}_i\r\}. 
\eea 
Remark also that the approximated increments are independent of the information process and of $\widehat{\theta}$. Now note that the general proof in \cite{vetter2015estimation} is conducted in the following two steps.

\begin{itemize}
    \item Compute an estimate for the deviations $ A_{i+k_n} - A_i - (\widetilde{A}_{i+k_n} - \widetilde{A}_{i} )$, $ B_{i+k_n} - B_i - (\widetilde{B}_{i+k_n} - \widetilde{B}_{i} )$, and $\widetilde{q}_i - \int_{t_i}^{t_{i+1}}\sigma_{s}^4ds$.
    \item Systematically use the previous estimate to replace $A_i$ (resp. $B_i$, $\widetilde{q}_i$) by its counterpart $\widetilde{A}_i$ (resp. $\widetilde{B}_i$, $\int_{t_i}^{t_{i+1}}\sigma_{s}^4ds$) in all the encountered expressions.
\end{itemize}
Since $\widetilde{A}_i$, $\widetilde{B}_i$ and $\int_{t_i}^{t_{i+1}}\sigma_{s}^4ds$ are independent of the information process and $\widehat{\theta}$, the second step holds in our setting as well with no modification in the proofs of \cite{vetter2015estimation}. Thus, all we have to do in order to prove the theorem is to adapt the first step replacing $A_i$, $B_i$ and $\widetilde{q}_i$ by $\widehat{A}_i$, $\widehat{B}_i$ and $\widehat{q}_i$. More precisely, we adapt Lemma A.1 and the second equation in the proof of (A.8) p. 2411 (corresponding to the approximation of $\widetilde{q}_i$ by $\int_{t_i}^{t_{i+1}}\sigma_{s}^4ds$) in \cite{vetter2015estimation} as follows (in the next lemma, recall that $\widetilde{A}_i$ and $\widetilde{B}_i$ depend on some parameter $p \geq 1$). 
\begin{lemma*}\label{lemmaVolofVol}
We have for any $r\geq1$, $p\geq1$, and any $i \in \{a_l(p), \cdots, a_l(p) + p k_n \}$ 
\beas 
\esp\l[\l|\widehat{A}_{i+k_n} - \widehat{A}_i-(\widetilde{A}_{i+k_n} - \widetilde{A}_{i} )\r|^r\r] \leq C(pn^{-1})^{r/2},
\eeas 
\beas
\esp\l[\l|\widehat{B}_{i+k_n} - \widehat{B}_i-(\widetilde{B}_{i+k_n} - \widetilde{B}_{i} )\r|^r\r] \leq C(pn^{-1})^{r/2},
\eeas 
\beas 
\esp\l[\l|\widehat{A}_{i+k_n} - \widehat{A}_i\r|^r\r] \leq Cn^{-r/2},
\eeas 
and
\beas 
\esp\l[\l|\widehat{B}_{i+k_n} - \widehat{B}_i\r|^r\r] \leq Cn^{-r/2}.
\eeas 
Moreover we have uniformly in $t \in [0,T]$
\beas 
\sqrt{\frac{n}{k_n}} \esp \l|\sum_{i=1}^{{[t/\Delta_n]-2k_n}}\frac{6}{k_n^2} \widehat{q}_i - \frac{6}{k_n^2}\int_0^t \sigma_s^4 ds  \r| = o(1). 
\eeas 
\end{lemma*}

\begin{proof}
By Lemma A.1 in \cite{vetter2015estimation}, it suffices to prove that we have 
\beas 
\esp\l[\l|\widehat{A}_{i+k_n} - \widehat{A}_i-( A _{i+k_n} -  A _{i} )\r|^r\r] \leq C(pn^{-1})^{r/2},
\eeas 
and a similar statement for $\widehat{B}_i$. Since $|\psi_{k}(\widehat{\theta})| \leq K/n $ for all $1\leq k \leq n$, we obtain
\beas 
\esp\l[\l|\widehat{A}_{i+k_n} - \widehat{A}_i-( A _{i+k_n} -  A _{i} )\r|^r\r] &\leq& \frac{2^rn^r}{k_n^r}\esp\l[\l| \sum_{j=1}^{k_n}\l\{\psi_{i+k_n+j}(\widehat{\theta})\Delta_{i+k_n+j}X -  \psi_{i +j}(\widehat{\theta})\Delta_{i +j}X\r\}\r|^r\r],\\
&\leq & \frac{ Cn^r}{k_n} \sum_{j=1}^{k_n} \esp  \l[\l|  \psi_{i+k_n+j}(\widehat{\theta})\Delta_{i+k_n+j}X\r|^r +  \l|\psi_{i +j}(\widehat{\theta})\Delta_{i +j}X \r|^r\r],\\
&\leq& \frac{ C }{k_n} \sum_{j=1}^{k_n} \underbrace{\esp \l[\l|   \Delta_{i+k_n+j}X\r|^r +  \l| \Delta_{i +j}X \r|^r\r]}_{\leq C n^{-r/2}},\\
&\leq& Cp^{r/2}n^{-r/2},
\eeas
since $p \geq 1$, where we used Jensen's inequality at the second step and the domination $|\psi_i(\widehat{\theta})|^r \leq C/n^r$ at the third step. Proving the other three inequalities and the approximation for $\widehat{q}_i$ can be done by similar calculation.
\end{proof}
Now, to prove Theorem \ref{thmVolofVol}, it is sufficient to follow closely the proof of Theorem 2.6 in \cite{vetter2015estimation} replacing all occurrences of $A_i$, $B_i$ and $\sum_{i=1}^{{[t/\Delta_n]-2k_n}}\frac{6}{k_n^2} \widehat{q}_i$ by $\widehat{A}_i$, $\widehat{B}_i$ and $\frac{6}{k_n^2}\int_0^t \sigma_s^4 ds $, and accordingly all applications of Lemma A.1 and the approximation for $\widehat{q}_i$ by Lemma \ref{lemmaVolofVol} above. 

\smallskip
A similar line of reasoning yields Corollary \ref{corVolofVol}.

\bibliography{biblio}
\bibliographystyle{spbasic}      

\end{document}